\documentclass[11pt]{article}

\usepackage{amsmath,amsfonts, amsthm, latexsym, amssymb,dsfont}
\usepackage{graphicx}
\usepackage{url}
\usepackage{bm}
\usepackage{color}
\usepackage{natbib}
\usepackage{mathrsfs}
\usepackage{epsfig}
\usepackage{float}
\usepackage{verbatim}
\usepackage[normalem]{ulem}
\usepackage{enumitem}

\definecolor{linkcolor}{rgb}{0.1,0,0.7}
\definecolor{urlcolor}{rgb}{1,0,0}
\usepackage[unicode=true,colorlinks,citecolor=linkcolor,linkcolor=blue,urlcolor=blue]{hyperref}

\allowdisplaybreaks[4]

\newtheorem{theorem}{Theorem}[section]
\newtheorem{lemma}[theorem]{Lemma}

\newtheorem{corollary}[theorem]{Corollary}
\newtheorem{proposition}[theorem]{Proposition}

\newtheorem{remark}[theorem]{Remark}

\title{Optimal consumption under a drawdown constraint over a finite horizon}

\author{
Xiaoshan Chen\thanks{School of Mathematical Science, South China Normal University, Guangzhou 510631, China, Email:\texttt{xschen@m.scnu.edu.cn}.}, \quad
Xun Li\thanks{Department of Applied Mathematics, The Hong Kong Polytechnic University, Hung Hom, Kowloon, Hong Kong. Email:\texttt{li.xun@polyu.edu.hk}.}, \quad
Fahuai Yi\thanks{School of Mathematical Science, South China Normal University, Guangzhou 510631, China, Email:\texttt{fhyi@scnu.edu.cn}.}, \quad
Xiang Yu\thanks{Department of Applied Mathematics, The Hong Kong Polytechnic University, Hung Hom, Kowloon, Hong Kong. Email:\texttt{xiang.yu@polyu.edu.hk}. Corresponding author.}}
\date{}

\begin{document}
\maketitle

\vspace{-0.4in}
\begin{abstract}
This paper studies a finite horizon utility maximization problem on excessive consumption under a drawdown constraint. Our control problem is an extension of the one considered in \cite{ABY19} to the model with a finite horizon and an extension of the one considered in \cite{JeonOh2022JMAA} to the model with zero interest rate. Contrary to \cite{ABY19}, we encounter a parabolic nonlinear HJB variational inequality with a gradient constraint, in which some time-dependent free boundaries complicate the analysis significantly. Meanwhile, our methodology is built on technical PDE arguments, which differs from the martingale approach in \cite{JeonOh2022JMAA}. Using the dual transform and considering the auxiliary variational inequality with gradient and function constraints, we establish the existence and uniqueness of the classical solution to the HJB variational inequality after the dimension reduction, and the associated free boundaries can be characterized in analytical form. Consequently, the piecewise optimal feedback controls and the time-dependent thresholds for the ratio of wealth and historical consumption peak can be obtained.

\vskip 0.2 true in \noindent {\bf Keyword:} Optimal consumption, drawdown constraint, parabolic variational inequality,
gradient constraint, free boundary.

\end{abstract}

\section{Introduction}
\setcounter{equation}{0}

Optimal portfolio and consumption via utility maximization has always been one of the core research topics in quantitative finance. Starting from the seminal works \cite{Merton69} and \cite{Mert1971JET}, a large amount of studies can be found in the literature by considering various incomplete market models, trading constraints, risk factors, etc. One notable research direction is to refine the measurement of consumption performance by encoding the impact of past consumption behavior. The so-called addictive habit formation preference recommends that the utility shall be generated by the difference between the current consumption rate and the historical weighted average of the past consumption. In addition, the infinite marginal utility mandates the addictive habit formation constraint that the consumption level needs to stay above the habit formation level, representing that the agent's standard of living can never be compromised. Along this direction, fruitful results can be found in various market models, see among \cite{constantinides1990habit}, \cite{DZ92}, \cite{schroder2002}, \cite{Munk08}, \cite{EK09}, \cite{Yu15}, \cite{Yu17}, \cite{YY22}, \cite{Bah2022}, \cite{Bo22}, \cite{Bah2023} and references therein.

Another rapidly growing research stream is to investigate the impact of the past consumption maximum instead of the historical average. The pioneering work \cite{Dy95} examines an extension of Merton's problem under a ratcheting constraint on consumption rate such that the consumption control needs to be non-decreasing. \cite{Arun12} later generalizes the model in  \cite{Dy95} by considering a drawdown constraint such that the consumption rate can not fall below a proportion of the past consumption maximum. \cite{ABY19} revisit the problem by considering a drawdown constraint on the excessive dividend rate up to the bankruptcy time, which can be regarded as an extension of the problem in \cite{Arun12} to the model with zero interest rate. \cite{JeonMMOR} further extend the work in \cite{Arun12} by considering general utility functions using the martingale duality approach where the dual optimal stopping problem is examined therein. \cite{Tanana} employs the general duality approach and establishes the existence of optimal consumption under a drawdown constraint in incomplete semimartingale market models. \cite{JeonOh2022JMAA} recently generalizes the approach in \cite{JeonMMOR} to the model with a finite horizon and addresses the existence of a solution to the dual optimal stopping problem. 

On the other hand, inspired by the time non-separable preference rooted in the habit formation preference, some recent studies also incorporate the past consumption maximum into the utility function as an endogenous reference level. \cite{GHR20} first adopt the Cobb-Douglas utility that is defined on the ratio between the consumption rate and the past consumption maximum and obtain the feedback controls in analytical form. \cite{DLPY} choose the same form of the linear habit formation preference and investigate an optimal consumption problem when the difference between the consumption rate and the past spending peak generates the utility. Later, the preference in \cite{DLPY} is generalized to an S-shaped utility in \cite{LiYuZ2021arXiv} to account for the agent's loss aversion over the relative consumption. \cite{LiangLuoY2022arXiv} extends the work in \cite{DLPY} by considering the change of risk aversion parameter concerning the reference level where an additional drawdown constraint is also required. \cite{LiYuZ2022} incorporates the dynamic life insurance control and expected bequest with an additional drawdown constraint on the consumption control.

As an important add-on to the existing literature, the present paper revisits the optimal excessive consumption problem under drawdown constraint as in \cite{ABY19}, however, with a finite investment horizon. We summarize the main contributions of the present paper as two-fold:

\begin{itemize}
\item[(i)] From the modelling perspective, it is well documented that heterogenous agents may have diverse choices of investment horizons in practice. In fact, as the agent's time horizon changes, the risk tolerance should be adjusted accordingly. Typically, agents seek more stable assets for short-term horizon and would call for a more aggressive strategy for the longer-term investment. Our model and research outcomes to incorporate a terminal horizon provide the flexibility to meet versatile needs in applications with different choices of investment horizons.

\hspace{0.3in}Similar to \cite{ABY19}, we adopt the classical power utility $U(x)=\frac{x^{1-p}}{1-p}$ with the relative risk aversion coefficient $0<p<1$. On one hand, many empirical studies suggest that the risk aversion of the agent should be relatively constant over wealth levels, which is an important merit of the power utility function. On the other hand, the power utility function enjoys the homogeneity, which enables us to reduce the dimension of the HJB equation by changing variables (see \eqref{changevb}). As a result, it is sufficient for us to focus on one state variable in studying the HJB equation and its dual PDE, which significantly simplifies the technical proofs to establish the regularity of the solution and the analytcial characterization of the time-dependent free boundary curves.        

\item[(ii)] From the methodology perspective, we encounter the associated parabolic HJB variational inequality with gradient constraint. It is well known that the global regularity of the parabolic problem requires a more delicate analysis of the time-dependent free boundaries and the smooth fit conditions. Some previous arguments in \cite{ABY19} crucially rely on the constant free boundary points as well as some explicit expressions of the value function, which are evidently not applicable in our framework due to time dependence. The present paper contributes to some new and rigorous proofs to establish the existence and the uniqueness of the classical solution to the associated HJB variational inequality \eqref{eq:VIU} (see Theorem \ref{thm:U}). Moreover, to verify that the dual transform is well defined, we can show that the solution $U$ is indeed increasing and concave in $\omega$ (see Lemma \ref{lem:conc}). 

\hspace{0.3in}We stress that, in many existing studies using the dual transform, see \cite{CCY12}, \cite{CY12}, \cite{DY09}, \cite{GYC19} among others, the dual domain is often the entire $\mathbb R^+\times(0,T]$, which facilitates some classical PDE arguments. In our framework, due to the drawdown constraint $c_t\geq \alpha z_t$ (see \eqref{drawcc} and \eqref{drawzz} for the description of the constraint), there exists a threshold for the ratio of the wealth level and the past consumption peak such that the minimum consumption plan needs to maintain at a subsistence level (see \eqref{solc*}) whenever the ratio falls  below that threshold. As a result, the left boundary $\omega=0$ is mapped to an unknown finite boundary $y_0(t)<\infty$, $t>0$. Employing the boundary condition at $y_0(t)$ in \eqref{eq:VIu}, we extend the variational inequality with a gradient constraint to an auxiliary variation inequality \eqref{eq:VIhatu} with both function and gradient constraints in the unbounded domain $\mathbb R^+\times(0,T]$. We can then apply several further transformations (see Propositions \ref{thm:v}, \ref{thm:tildeu} and \ref{thm:u}) and modify some technical arguments in \cite{CLY19} and \cite{CY12} to establish the existence and uniqueness of the solution in $C^{2,1}$ regularity to the primal parabolic HJB variational inequality (see Theorem \ref{thm:U}). More importantly, we provide some tailor-made and technical arguments to characterize the associated time-dependent free boundaries in analytical form such that the smooth fit conditions hold (see Theorem \ref{thm:Ufb}).
\end{itemize}

With the help of our new PDE results, we can derive the optimal consumption and portfolio in piecewise feedback form and identify all analytical time-dependent thresholds for the ratio between the wealth level and the past consumption maximum, dividing the domain into three regions for different consumption behaviors (see Theorem \ref{thm:V}). In addition, the analytical threshold functions allow us to theoretically verify their quantitative dependence on the constraint parameter $\alpha$.

The rest of the paper is organized as follows. In Section \ref{sec:model}, we introduce the market model and the utility maximization problem under a consumption drawdown constraint. By dimension reduction using the homogeneity of power utility, we formulate the associated HJB variational inequality and its dual problem. In Section \ref{sec:dual}, we first analyze the dual linear parabolic variational inequality by considering some auxiliary problems with gradient and function constraints. By showing the existence and uniqueness of the solution to the auxiliary problems and characterizing their free boundaries, we obtain the unique classical solution to the dual HJB variational inequality. In Section \ref{sec:control}, using the results from the dual problem, we establish the unique classical solution to the primal HJB variational inequality and verify the optimal feedback controls and all associated time-dependent thresholds in analytical form. In Section \ref{sec:con}, we conclude our theoretical contributions and discuss some future research directions.

\section{Market Model and Problem Formulation}\label{sec:model}
\setcounter{equation}{0}
\subsection{Model setup}
Let $(\Omega, \mathcal{F}, \mathbb{F}, \mathbb{P})$ be a standard filtered probability space, where $\mathbb{F}=(\mathcal{F}_t)_{t\in [0,T]}$ satisfies the usual conditions. We consider a financial market consisting of one riskless asset and one risky asset, and the terminal time horizon is denoted by $T$. The riskless asset price satisfies $dB_t=rB_tdt$ where $r\geq 0$ represents the constant interest rate. The risky asset price follows the dynamics
\begin{equation}
dS_t=S_t(\mu+r) dt+S_t\sigma dW_t,\ \ t\in [0,T],\nonumber
\end{equation}
where $W$ is an $\mathbb{F}$-adapted Brownian motion and the drift $(\mu+r)$ and volatility $\sigma>0$ are given constants. It is assumed that the excessive return $\mu>0$, i.e., the risky asset's return is higher than the interest rate.

Let $(\pi_t)_{t\in [0,T]}$ represent the dynamic amount that the investor allocates in the risky asset and $(C_t)_{t\in [0,T]}$ denote the dynamic consumption rate by the investor. In this paper, we consider a drawdown constraint on the excess consumption rate $c_t:=C_t-rX_t$ in the sense that $c_t$ cannot go below a fraction $\alpha\in(0,1)$ of its past maximum that
\begin{align}\label{drawcc}
c_t\geq\alpha z_t, \quad t\in [0,T].
\end{align}

Here, the non-decreasing reference process $(z_t)_{t\in [0,T]}$ is defined as the historical excessive spending maximum
\begin{align}\label{drawzz}
z_t=\max\Big\{z,\ \sup\limits_{s\leq t}c_s\Big\},
\end{align}
and $z\geq 0$ is the initial reference level. 

In our setting, $rX_t$ stands for the subsistence consumption level that may refer to some mandated daily expenses in practice. On top of that, the drawdown constraint \eqref{drawcc} is imposed to reflect some practical situations that some decision makers may psychologically feel painful when there is a decline in the consumption plan comparing with the past spending peak. In particular, some large expenditures not only spur some long term continuing spending such as maintenance and repair, but also lift up the investor’s standard of living that affects the future consumption decisions. Moreover, by interpreting the consumption control as dividend payment, \cite{ABY19} also discussed another motivation of the drawdown constraint that some shareholders feel disappointment when there is a decline in dividend policies.

The self-financing wealth process $X_t$ is then governed by the SDE that
\begin{align*}
dX_t = (r(X_t-\pi_t)-C_t)dt + \pi_t\left((\mu+r)dt+\sigma dW_t\right),
\end{align*}
which can be equivalently written by
\begin{align}
dX_t= (\mu\pi_t-c_t)dt+\sigma\pi_tdW_t,\ \ t\in[0,T],\label{wealthpro} 
\end{align}
with the initial wealth $X_0=x\geq 0$. 

Let $\mathcal{A}(x)$ denote the set of admissible controls $(\pi_t, c_t)$ if $(\pi_t)_{t\in[0,T]}$ is $\mathbb{F}$-progressively measurable, $(c_t)_{t\in[0,T]}$ is $\mathbb{F}$-predictable, the integrability condition $\mathbb{E}[\int_0^T (c_t+\pi_t^2)dt]<+\infty$ holds and the drawdown constraint $c_t\geq\alpha z_t$ is satisfied a.s. for all $t\in[0,T]$.

The goal of the agent is to maximize the expected utility of the excessive consumption rate up to time $T\wedge \tau$ under the drawdown constraint, where $\tau$ is the bankruptcy time defined by $\tau:=\inf\big\{t\geq0 ~ | ~ X_t\leq 0\big\}$.

To embed the control problem into a Markovian framework and facilitate the dynamic programming arguments, we choose the consumption running maximum process $z_t$ in \eqref{drawzz} as the second state process. The value function of the stochastic control problem over a finite time horizon is given by
\begin{eqnarray}\label{eq:Valuefunction}
V(x,z,t)=\sup\limits_{(\pi,c)\in\mathcal{A}(x)}\mathbb{E}\left[\int_t^{T\wedge\tau}e^{-\delta (s-t)}\frac{c_s^{1-p}}{1-p}ds+e^{-\delta (T\wedge \tau-t)}\frac{X_{T\wedge \tau}^{1-p}}{1-p}\Bigg|X_t=x, z_t=z\right],
\end{eqnarray}
where the constant $\delta\geq 0$ denotes the subjective time preference parameter, and $0<p<1$ stands for the agent's relative risk aversion coefficient.

We first note that the problem \eqref{eq:Valuefunction} is an extension of the control problem considered in \cite{ABY19} to the model with a finite horizon. Some new mathematical challenges arise as we encounter a parabolic variational inequality with time-dependent free boundaries. One main contribution of the present paper is that all time-dependent free boundaries stemming from the control constraint $\alpha z_t\leq c_t\leq z_t$ can be characterized in analytical form (see \eqref{solvo*}, \eqref{solvo1} and \eqref{solvoa} in Thereom \ref{thm:Ufb}), allowing us to identify time-dependent thresholds for the ratio between the wealth level and the past consumption running maximum to choose among different optimal feedback consumption strategies, see \eqref{solc*} in Theorem \ref{thm:V}.  

On the other hand, if we interpret $c_t$ in problem \eqref{eq:Valuefunction} as a control of consumption rate and regard $X_t$ in \eqref{wealthpro} as the resulting wealth process under the control pair $(\pi_t, c_t)$, then the control problem \eqref{eq:Valuefunction} is actually equivalent to a finite horizon optimal consumption problem under a drawdown constraint in the financial market with zero interest rate (as $r=0$ in the SDE \eqref{wealthpro} of $X_t$ under the control $(\pi_t,c_t)$). The finite horizon optimal consumption problem under a drawdown constraint has been studied by \cite{JeonOh2022JMAA}, however, under the crucial assumptions that the interest rate $r>0$ and the initial wealth is sufficiently large that $X_t>\alpha z_{t-}\frac{1-e^{-r(T-t)}}{r}$ (see Assumption 3.1 in \cite{JeonOh2022JMAA}). Based on the martingale and duality approach, the original control problem in \cite{JeonOh2022JMAA} is transformed into an infinite series of optimal stopping problems, which requires the optimal wealth process to be strictly positive that $X_t>0$ and hence $X_t>\alpha z_{t-}\frac{1-e^{-r(T-t)}}{r}$ has to be mandated as an upper bound on $z_t$. 

In contrast, we have interest rate $r=0$ in the equivalent optimal consumption problem, and no constraint on $X_t$ or $z_t$ is imposed because we allow the wealth process to hit zero on or before the terminal horizon $T$ and will terminate the investment and consumption control once the bankruptcy occurs. Therefore, the main results and the approach in \cite{JeonOh2022JMAA} can not cover our problem \eqref{eq:Valuefunction} with $r=0$ and $x\geq 0$. Instead, we rely on the technical analysis of the HJB variational inequality. Note that the interest rate $r$ does not appear in SDE \eqref{wealthpro} and the objective function in \eqref{eq:Valuefunction}, our main results, especially all associated wealth thresholds, do not depend on $r$. Some major challenges are characterizing the free boundaries caused by the control constraint  $\alpha z_t\leq c_t\leq z_t$ and the global regularity of the solution to the HJB variational inequality. 

 \subsection{The control problem and main result}

By heuristic dynamic programming arguments and the martingale optimality condition, we note that the term $\partial_zV=0$ holds whenever the monotone process $z_t$ is strictly increasing, and the associated HJB variational inequality can be written as
\begin{equation}\label{HJBV}
\left\{\begin{array}{l}
\max\Big\{\sup\limits_{\substack{\pi\in\mathbb{R},\\ \alpha z\leq c\leq z}}\left[\partial_tV+\frac 1 2\sigma^2\pi^2\partial_{xx}V+(\mu\pi-c)\partial_xV-\delta V+\frac{c^{1-p}}{1-p}\right], \partial_zV\Big\}=0,\ (x,z,t)\in Q, \\
V(0,z,t)=0, \quad\quad z>0,t\in[0,T), \\
V(x,z,T)=\frac{x^{1-p}}{1-p}, \quad\quad x\geq0,\;z>0,
\end{array}\right.
\end{equation}
where $Q:=(0,+\infty)\times(0,+\infty)\times[0,T).$

It is straightforward to see that the value function $V(x,z,t)$ in \eqref{eq:Valuefunction} is homogeneous of degree $1-p$ with respect to $x$ and $z$ such that
$V(\beta x,\beta z,t)=\beta^{1-p} V(x,z,t)$. As a result, we can consider the change of variable 
\begin{align}\label{changevb}
\omega:=\frac{x}{z}\geq 0,
\end{align}
and reduce the dimension that
\begin{equation}\label{eq:VUdef}
V(x,z,t)=z^{1-p}V\Big(\frac{x}{z},1,t\Big) =: z^{1-p}U\Big(\frac{x}{z},t\Big)=z^{1-p}U\Big(\omega,t\Big).
\end{equation}
It then follows that
\begin{eqnarray}\label{eq:VU}
\left\{\begin{array}{ll}
\partial_tV = z^{1-p}\partial_tU, \\
 \partial_xV = z^{1-p}(\frac{1}{z})\partial_\omega U = z^{-p}\partial_\omega U, \\
 \partial_{xx}V = z^{-(1+p)}\partial_{\omega\omega}U, \\
\partial_zV = (1-p)z^{-p}U+z^{1-p}(-\frac{x}{z^2})\partial_{\omega}U = z^{-p}[(1-p)U-\omega \partial_{\omega}U].
\end{array}\right.
\end{eqnarray}
Moreover, let us consider the auxiliary controls $\hat{\pi}(\omega, t)=\frac{\pi(x,z,t)}{z}$ and $\hat{c}_t(\omega, t)=\frac{c(x,z,t)}{z}$. The HJB equation \eqref{HJBV} can be rewritten as
\begin{eqnarray}\label{eq:HJBU}
\left\{\begin{array}{l}
\max\Big\{\partial_tU+\sup\limits_{\hat\pi\in\mathbb R}\left\{\frac 1 2\sigma^2\hat\pi^2\partial_{\omega\omega}U+\mu \hat\pi\partial_{\omega}U\right\}+\sup\limits_{ \alpha\leq\hat c\leq1}\left\{\frac{\hat c^{1-p}}{1-p}-\hat c\partial_{\omega}U\right\}-\delta U, \\
\hfill (1-p)U-\omega\partial_{\omega}U\Big\} = 0, \quad (\omega,t)\in\mathcal Q,\\
U(0,t) = 0, \quad t\in[0,T),\\
U(\omega,T) = \frac{\omega^{1-p}}{1-p}, \quad \omega\geq0,
\end{array}\right.
\end{eqnarray}
where $\mathcal Q:=\mathbb R^+\times [0,T)$.

We first present the main result of the paper, and its proof is deferred to Section \ref{sec:control}.

\begin{theorem}[Verification Theorem]\label{thm:V}

There exists a unique classical solution $U(\omega,t)\in C^{2,1}(\mathcal Q)$ to problem \eqref{eq:HJBU}, and $V(x,z,t):=z^{1-p}U(\frac{x}{z},t) )\in C^{2,1}(Q)\cap C(\bar Q) $ is the unique classical solution to problem \eqref{HJBV}. Moreover, we have that
\begin{eqnarray}\label{eq:pxxV}
\partial_{xx}V(x,z,t)<0, \quad (x,z,t) \in Q,
\end{eqnarray}
The optimal feedback controls of the problem \eqref{eq:Valuefunction} are given by
\begin{eqnarray}
&&\pi^*(x,z,t)=z\hat \pi^*\left(\frac{x}{z}, t\right)=z\left[-\frac{\mu}{\sigma^2}\frac{\partial_\omega U (\frac{x}{z}, t)}{\partial_{\omega\omega}U(\frac{x}{z}, t)}\right],\label{solpi*} \\
&&c^*(x,z,t)=z{\hat c}^*\left(\frac{x}{z},t\right)=\left\{\begin{array}{ll}
\alpha z, & \mbox{if }  0 < x \leq \omega_\alpha(t)z, \\
(\partial_\omega U)^{-\frac{1}{p}}\left(\frac{x}{z}, t\right)z, & \mbox{if } \omega_\alpha(t)z< x < \omega_1(t)z, \\
z, & \mbox{if } \omega_1(t)z \leq x \leq \omega^*(t)z, \\
\end{array}\right. \label{solc*} \\
 &&\mathcal D=\{x,z,t)\in Q ~ | ~ \partial_zV(x,z,t)=0\}= \{(x,z,t)\in Q ~ | ~ x\geq\omega^*(t)z\}, \label{fb:V1} \\
 &&\mathcal C=\{x,z,t)\in Q ~ | ~ \partial_zV(x,z,t)<0\}=\{(x,z,t)\in Q ~ | ~ x<\omega^*(t)z\}, \label{fb:V2}
\end{eqnarray}
where $ \omega_\alpha(t), \omega_1(t)$ and $\omega^*(t)$ are free boundaries to problem \eqref{eq:VIU}, which are characterized analytically in Theorem \ref{thm:Ufb}. 
\end{theorem}

\begin{remark}\label{rmk:omega} In what follows, we first elaborate the intuition behind the existence of the free boundary $\omega^*(t)$ such that $c^*(x,z,t)=z$ for $\omega_1(t)z\leq x\leq \omega^*(t)z$ in Theorem \ref{thm:V}.

In Section 4 of \cite{JeonOh2022JMAA}, it is shown by duality and martingale approach that there exists an optimal adjustment boundary $\omega^*(t)$ such that: (i) if $\alpha\frac{1-e^{-r(T-t)}}{r} X^*_t/z^*_{t}<\omega^*(t)$, the optimal consumption satisfies $\alpha z_t^*\leq c^*_t\leq z_{t}^*$; (ii) if $X^*_t/z^*_{t}\geq \omega^*(t)$, the resulting running maximum process $z_t^*$ is increasing and if $X^*_t/z^*_{t-}>\omega^*(t)$, the optimal consumption can be chosen in the form of $c_t^*=X^*_{t}/\omega^*(t)>z^*_{t-}$ such that the resulting $z^*_t$ immediately jumps from $z^*_{t-}$ to a new global maximum level $z^*_t=X^*_{t}/\omega^*(t)$. As a result, the free boundary $\omega^*(t)$ can be used to split the whole domain of $(x,z,t)$ into two regions connected in $t$ variable and the controlled two dimensional process $(X_t^*, z_t^*)$ only diffuses within the region $\{ X_t^*\leq z_t^* \omega^*(t)\}$ for any $0<t\leq T$. The only possibility for $X_t^*>z_{t-}^* \omega^*(t)$ to occur is at the initial time $t=0$, at which instant the process $z_{0-}^*=z$ jumps immediately to $x/\omega^*(0)$.

Motivated by the result in \cite{JeonOh2022JMAA}, we conjecture and will verify later in our model that there also exists such a time-dependent free boundary $\omega^*(t)$, which is the critical threshold of the wealth-to-consumption-peak ratio under the optimal control $(\pi^*,c^*)$ such that if $X^*_{t}>\omega^*(t)z^*_{t-}$, the resulting $z^*_t$ immediately jumps from $z^*_{t-}$ to a new global maximum level $z^*_t=X^*_{t}/\omega^*(t)$. If $X^*_t\leq\omega^*(t)z^*_t$, the agent chooses the excessive consumption rate staying in $[\alpha z^*_t, z^*_t]$. 

Given this conjectured free boundary $\omega^*(t)$, the domain $Q=(0,+\infty)\times(0,\infty)\times [0,T)$ can be split into two regions $\mathcal{D}:=\{(x,z,t)\in Q| x\geq \omega^*(t)z\}$ and $\mathcal{C}:=\{(x,z,t)\in Q| x<\omega^*(t)z\}$, which are connected in the time variable $t$.  We will first study the existence of a classical solution to the HJB variational inequality with the time-dependent free boundary $\omega^*(t)$ and characterize $\omega^*(t)$ in the analytical form in Theorem \ref{thm:Ufb}. Building upon the classical solution to the HJB variational inequality, we then derive the piecewise feedback functions for the optimal control. Finally, the verification theorem on the optimal control guarantees the validity of our conjecture and the existence of such a free boundary $\omega^*(t)$ for the optimal control $(\pi^*, c^*)$. We note that a similar hypothesis on the existence of a constant free boundary $\omega^*$ (independent of $t$) is also made and verified in \cite{ABY19} for the infinite horizon stochastic control problem using the verification theorem on the optimal control.

\end{remark} 
 
Based on the variational inequality \eqref{HJBV} and the main results in Verification Theorem \ref{thm:V}, we plot the numerical illustrations of the value function, the optimal feedback functions of portfolio and consumption in terms of the wealth variable $x$ while fixing $z=1$ and $t=0.5$ in Figure \ref{fig-V} as below:

\begin{figure}[H]
    \centering
\includegraphics[width=4.5cm]{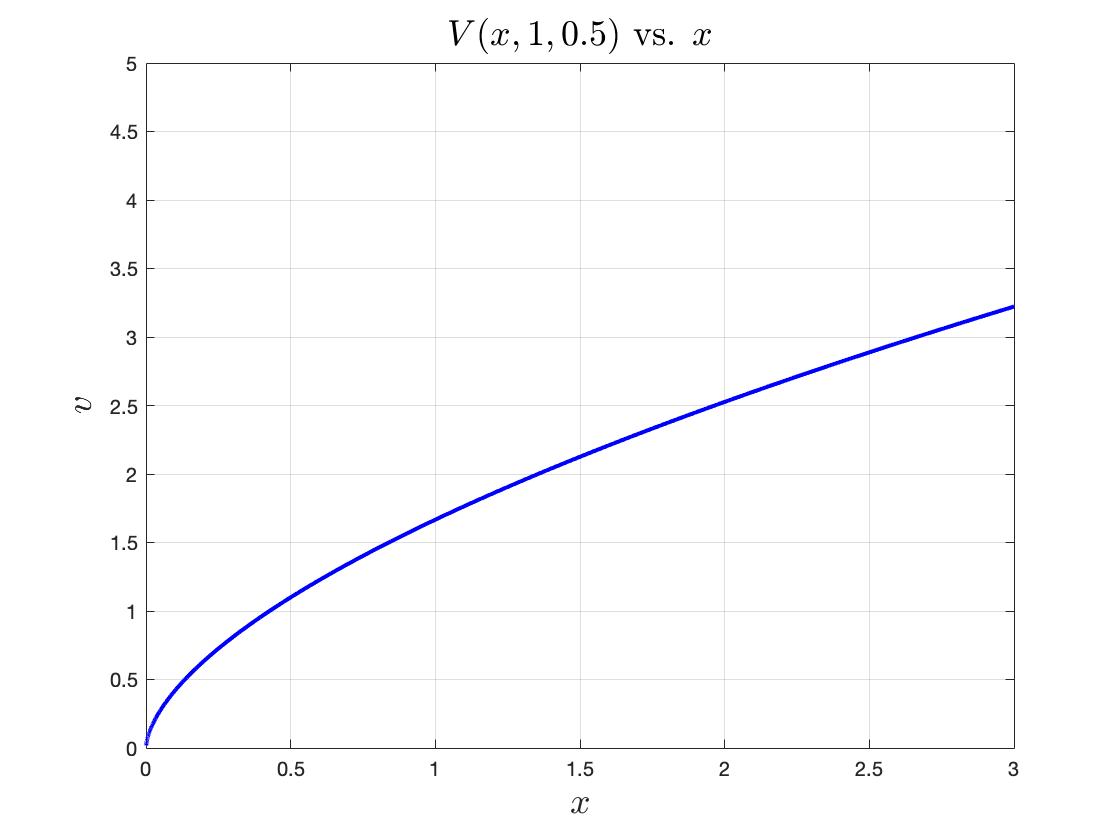}\quad
\includegraphics[width=4.5cm]{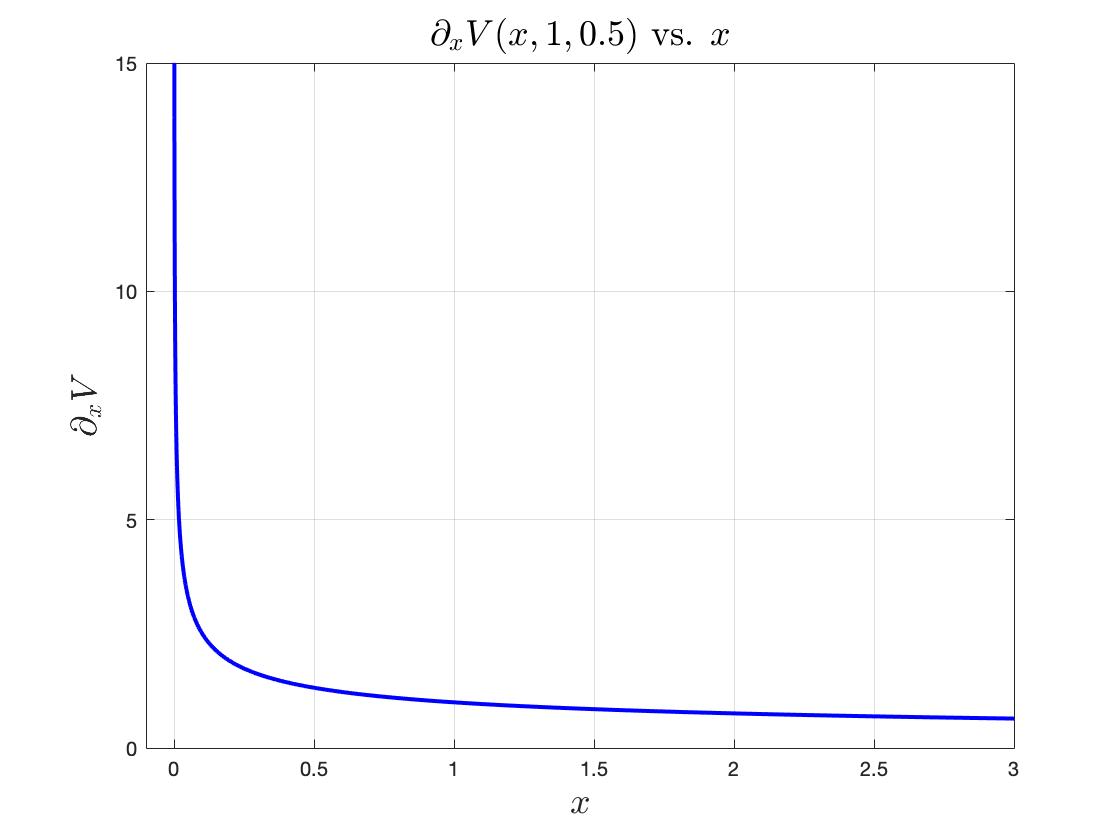}\quad
\includegraphics[width=4.5cm]{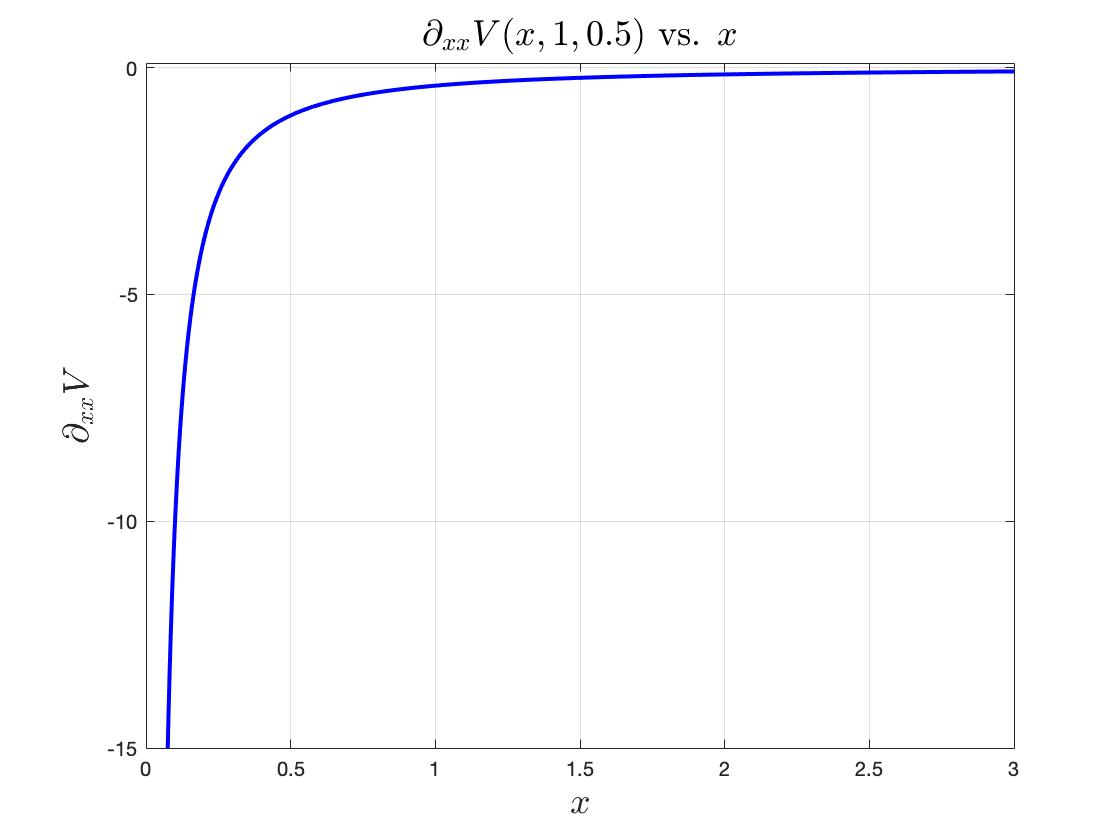}\\
\includegraphics[width=5cm]{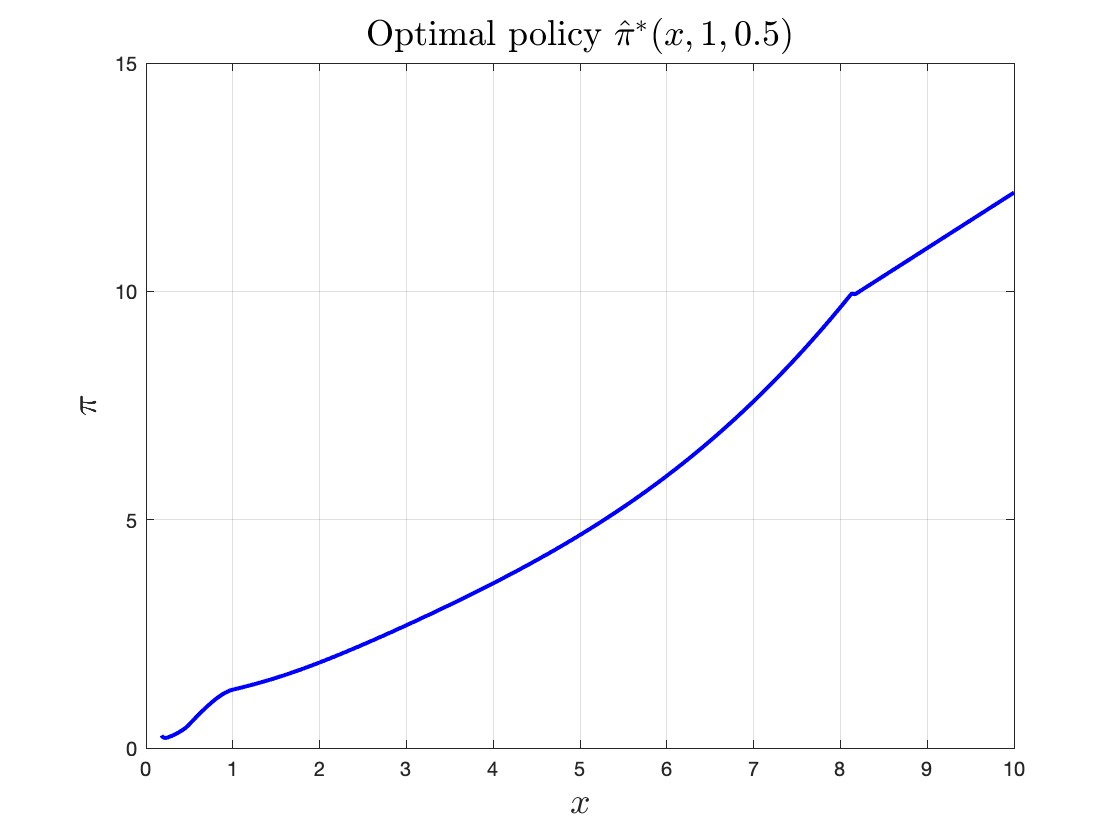}\quad
\includegraphics[width=5cm]{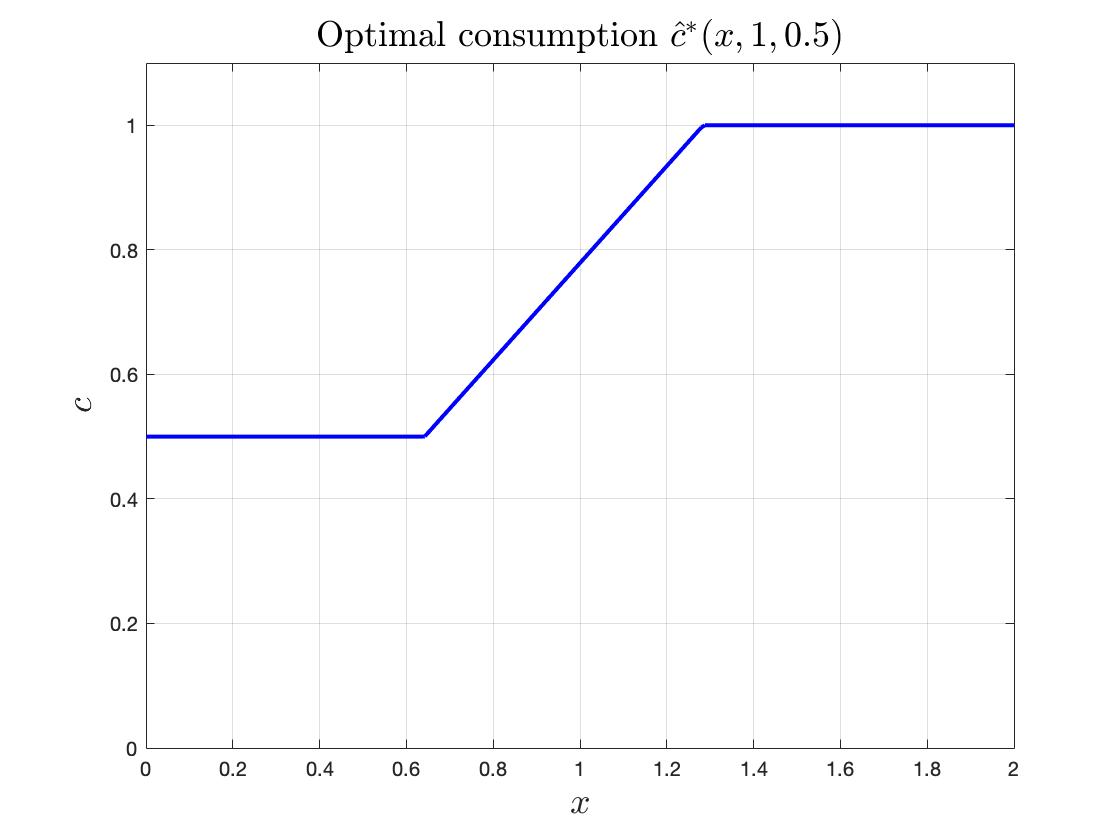}
\caption{The numerical illustration of the value function and its derivative with respect to $x$ (top panel)
The numerical illustration of the optimal portfolio and the optimal consumption with respect to $x$ (bottom panel)}\label{fig-V}
\end{figure}

As shown in Figure \ref{fig-V}, the value function is strictly increasing and concave in the wealth variable $x$. More importantly, both optimal feedback functions of portfolio and consumption rate are increasing in $x$. However, comparing with the Merton's solution, the optimal portfolio is no longer a constant proportion strategy of the wealth level, in fact, it is even not simply concave or convex in the wealth variable depending on the ratio of the wealth and the past consumption peak. The right-bottom panel also illustrates the piecewise consumption behavior in Theorem \ref{thm:V}: when $x/z\leq \omega_{\alpha}(t)$, the optimal consumption $\hat{c}^*(\frac{x}{z}, t)=\alpha$; when $\omega_{\alpha}(t)<x/z<\omega_1(t)$, the optimal consumption is the first order condition $\hat{c}^*(\frac{x}{z}, t)=(\partial_{\omega}U)^{-\frac{1}{p}}(\frac{x}{z}, t)$; and when $\omega_1(t)\leq x/z\leq \omega^*(t)$, the optimal consumption is equal to the historical consumption peak.

\subsection{The dual variational inequality}
We plan to employ the dual transform to linearize the HJB variational inequality  \eqref{eq:HJBU} and study the existence of its classical solution. To this end, we first show below that the dual transform is well defined.  

%

\begin{lemma}\label{lem:conc}
The solution
$U(\omega,t)$ of the problem \eqref{eq:HJBU} satisfies
\begin{eqnarray}
&&\partial_{\omega\omega}U(\omega,t)<0,\quad (\omega,t)\in\mathcal Q,\label{eq:Uconc}\\
&&\partial_{\omega}U(\omega,t)>0,\quad (\omega,t)\in\mathcal Q.\label{eq:Umono}
\end{eqnarray}
\end{lemma}

\begin{proof}
From the transformation \eqref{eq:VUdef} between $V(x,z,t)$ and $U(\omega,t)$, let us first show the concavity of $U(\omega,t)$. Suppose that there exists a point $(\omega_0,t_0)\in \mathcal Q$ such that $\partial_{\omega\omega}U(\omega_0,t_0)>0$, then
$$\sup\limits_{}\left\{\frac12\sigma^2\tilde\pi^2\partial_{\omega\omega}U(\omega_0,t_0)+\mu\tilde\pi\partial_{\omega}U(\omega_0,t_0)\right\}=+\infty,$$
which is a contradiction to \eqref{eq:HJBU}, hence $\partial_{\omega\omega}U(\omega,t)\leq 0$ in $\mathcal Q$.

Next, if there exists a point $(\omega_0,t_0)\in \mathcal Q$ such that $\partial_{\omega\omega}U(\omega_0,t_0)=0$ and $\partial_{\omega}U(\omega_0,t_0)\neq0$, then
$$\sup\limits_{\tilde\pi}\left\{\frac12\sigma^2\tilde\pi^2\partial_{\omega\omega}U(\omega_0,t_0)+\mu\tilde\pi\partial_{\omega}U(\omega_0,t_0)\right\}=\sup\limits_{\tilde\pi}\{\mu\tilde\pi\partial_{\omega}U(\omega_0,t_0)\}=+\infty,$$
which also leads to a contradiction. Therefore, if there exists a point $(\bar\omega,\bar t)$ such that $\partial_{\omega\omega}U(\bar\omega,\bar t)=0$, we must have $\partial_{\omega}U(\bar\omega,\bar t)=0$, thus \eqref{eq:HJBU} becomes
\begin{equation}\label{eq:1}
\max\left\{\partial_tU(\bar\omega,\bar t)+\frac{1}{1-p}-\delta U(\bar\omega,\bar t),(1-p)U(\bar\omega,\bar t)\right\}=0,
\end{equation}
which yields that $(1-p)U(\bar\omega,\bar t)\leq 0$. Moreover, by the definition \eqref{eq:Valuefunction} of $V(x,z,t)$ and the transformation \eqref{eq:VUdef} between $V(x,z,t)$ and $U(\omega,t)$, we know $U\geq0$ in $\mathcal Q$, thus $U$ attains the minimum at the point $(\bar\omega,\bar t)$. It then follows that $\partial_tU(\bar\omega,\bar t)\geq0$, and
$$
\partial_tU(\bar\omega,\bar t)+\frac{1}{1-p}-\delta U(\bar\omega,\bar t)\geq\frac{1}{1-p},
$$
which contradicts \eqref{eq:1}. In summary, we obtain \eqref{eq:Uconc}.

We next show the strict monotonicity of $U(\omega,t)$. By the definition \eqref{eq:Valuefunction} of $V(x,z,t)$ and the transformation \eqref{eq:VUdef}, we know that $U\geq0$ in $\mathcal Q$. Combining with the inequality $(1-p)U-\omega\partial_{\omega}U\leq0$ in \eqref{eq:HJBU}, we have
$$
\partial_{\omega}U\geq (1-p)U\geq0,\quad (\omega,t)\in\mathcal Q.
$$
Moreover, by the concavity of $U(\omega,t)$, we know $\partial_{\omega}U$ is strictly decreasing in $\omega$. Hence, if there exists a point $(\omega_0,t_0)\in \mathcal Q$ such that $\partial_{\omega}U(\omega_0,t_0)=0$, it holds that
$$
\partial_{\omega}U(\omega,t_0)<0,\quad \omega>\omega_0,
$$ 
which yields a contradiction. Therefore, we have that \eqref{eq:Umono} holds.
\end{proof}

As $U$ is increasing and concave in $\omega$, we can choose the candidate optimal feedback control ${\hat \pi}^*(\omega, t)$ by the first order condition that
$$\hat\pi^*(\omega, t):=-\frac{\mu}{\sigma^2}\frac{\partial_{\omega}U}{\partial_{\omega\omega}U}\geq0.$$
Considering the constraint $\alpha\leq \hat{c}\leq 1$, we can choose the candidate optimal feedback control 
\begin{equation}\label{eq:defc}
\hat c^*(\omega, t):=\max\Big\{\alpha,\min\big\{1, (\partial_{\omega}U)^{-\frac{1}{p}}\big\}\Big\}.
\end{equation}
Then \eqref{eq:HJBU} becomes
\begin{equation}\label{eq:VIU}
\left\{\begin{array}{l}
\max\left\{ \partial_{t}U-\frac 1 2\frac{\mu^2}{\sigma^2}\frac{( \partial_{\omega}U)^2}{ \partial_{\omega\omega} U}+\frac{(\hat{c}^*(\omega,t))^{1-p}}{1-p}-{\hat c}^*(\omega,t) \partial_\omega U-\delta U,(1-p)U-\omega  \partial_\omega U\right\} = 0, \\
\hfill (\omega,t)\in\mathcal Q, \\
U(0,t) = 0, \quad\quad t\in[0,T), \\
U(\omega,T) = \frac{\omega^{1-p}}{1-p}, \quad \omega\geq0.
\end{array}\right.
\end{equation}
%

For the conjectured free boundary $\omega^*(t)$ in Remark \ref{rmk:omega} under the optimal control $(\pi^*, c^*)$, using the relationship between $V(x,z,t)$ and $U(\omega, t)$, $\mathcal Q$ can also be divided into two regions, namely the continuation region and the jump region denoted by $CR$ and $JR$ (see the illustration in Fig. 1) that
\begin{eqnarray}
&&CR:=\big\{(\omega,t) | (1-p)U-\omega  \partial_\omega U<0\big\}=\big\{(\omega,t) | \omega< \omega^*(t)\big\}, \\
&&JR:=\big\{(\omega,t) | (1-p)U-\omega  \partial_\omega U=0\big\}=\big\{(\omega,t) | \omega\geq \omega^*(t)\big\}.
\end{eqnarray}
Later, we will rigorously characterize $\omega^*(t)$ and two regions $CR$ and $JR$ in analytical form in Theorem  \ref{thm:Ufb}.


To tackle the nonlinear parabolic variational inequality \eqref{eq:VIU}, we employ the convex dual transform that
$$u(y,t)=\max\limits_{\omega>0}[U(\omega,t)-\omega y].$$
As $\partial_{\omega\omega}(U(\omega,t)-\omega y) = \partial_{\omega\omega}U < 0$, $U(\omega,t)-\omega y$ is concave in $\omega$. It implies that the critical value $\omega_y$ satisfies
\begin{equation}\label{eq:yomega}
\partial_\omega U(\omega_y,t)=y>0.
\end{equation}
By $\partial_{\omega\omega}U < 0$, there exists $I(y,t)$, the inverse of $\partial_\omega U$ in $\omega$ such that
\begin{equation}\label{eq:omegay}
\omega_y=(\partial_\omega U)^{-1}(y,t)=I(y,t)>0.
\end{equation}
Then
\begin{equation}\label{eq:defu}
u(y,t)=U(I(y,t),t)-yI(y,t),
\end{equation}
which leads to
\begin{eqnarray}
&&\partial_y u(y,t) = \partial_\omega U(I(y,t),t)\partial_yI(y,t)-I(y,t)-y\partial_yI(y,t)=-I(y,t)<0,\label{eq:uy}\\
&&\partial_{yy}u(y,t)=-\partial_yI(y,t)=-\frac{1}{\partial_{\omega\omega}U(I(y,t),t)}>0.\label{eq:uyy}
\end{eqnarray}
Hence, $I(y,t)$ strictly decreases in $y$, $u(y,t)$ strictly decreases and is convex in $y$. Moreover, it follows from \eqref{eq:defu} that we have
\begin{eqnarray}\label{eq:ptU}
\partial_t u(y,t) = \partial_\omega U(I(y,t),t)\partial_t I(y,t) + \partial_t U(I(y,t),t)-y\partial_tI(y,t) = \partial_t U(I(y,t),t).
\end{eqnarray}
In addition, define $y_0(t) := \partial_\omega U(0,t)$, which implies that
\begin{equation}\label{eq:y0t}
\partial_yu(y_0(t),t)=-I(y_0(t),t)=0.
\end{equation}
Combining \eqref{eq:defu} with the boundary condition \eqref{eq:y0t}, we obtain
$$u(y_0(t),t)=U(I(y_0(t),t),t)-y_0(t)I(y_0(t),t)=U(0,t)=0.$$
Therefore, we have
$$U(\omega,t) = \min\limits_{y\in(0,y_0(t)]}[u(y,t)+y\omega],$$
Let $y_\omega$ be the critical value satisfying $\partial_y u(y_\omega,t)+\omega=0$, it holds that
\begin{equation}\label{eq:Uu}
U(\omega,t)=u(y_\omega,t)+\omega y_\omega.
\end{equation}
The linear dual variational inequality of \eqref{eq:VIU} can be written as
\begin{equation}\label{eq:VIu}
\left\{\begin{array}{l}
\max\left\{\partial_t u + \frac{1}{2}\frac{\mu^2}{\sigma^2}y^2\partial_{yy}u + \delta y\partial_y u - \delta u - f(y), (1-p)u + py \partial_y u\right\} = 0, \\
\hfill (y,t)\in (0,y_0(t))\times(0,T], \\
u(y_0(t),t)=\partial_yu(y_0(t),t)=0, \\
u(y,T)=\frac{p}{1-p}y^{1-\frac{1}{p}},\quad y\geq0,
\end{array}\right.
\end{equation}
where
\begin{eqnarray*}
f(y)=\left\{\begin{array}{ll}
\alpha y-\frac{\alpha^{1-p}}{1-p},  & \mbox{if } y\geq \alpha^{-p}, \\
-\frac{p}{1-p}y^{1-\frac{1}{p}},       & \mbox{if } 1< y<\alpha^{-p}, \\
y-\frac{1}{1-p},                              & \mbox{if } y\leq1,
\end{array}\right.
\end{eqnarray*}
which is continuously differentiable in $y$. Moreover, we define $y^*(t) = \partial_{\omega}U(\omega^*(t),t)$.

\section{The Solution to the Dual Variational Inequality \eqref{eq:VIu}}\label{sec:dual}
\setcounter{equation}{0}

\subsection{Auxiliary dual variational inequality}
Taking advantage of the boundary condition of $u(y,t)$ on $y_0(t)$, we can expand the solution $u(y,t)$ to the problem \eqref{eq:VIu} from $(0,y_0(t))\times (0,T]$ to the enlarged domain $\mathcal Q=\mathbb R^+\times [0,T)$. Let us consider $\hat u(y,t)$ that satisfies the auxiliary variational inequality on $\mathcal Q$ that
\begin{equation}\label{eq:VIhatu}
\left\{\begin{array}{l}
\max\left\{\partial_t\hat u+\frac 1 2\frac{\mu^2}{\sigma^2}y^2\partial_{yy}\hat u + \delta y\partial_y\hat u-\delta\hat  u-f(y),\ (1-p)\hat u+py\partial_y\hat u,\ -\hat  u\right\}=0,\quad (y,t)\in\mathcal Q,\\
\hat u(y,T)=\frac{p}{1-p}y^{1-\frac{1}{p}},\quad y\geq0.
\end{array}\right.
\end{equation}
It follows that
$$u(y,t)=\hat u(y,t),\quad (y,t)\in(0,y_0(t))\times (0,T].$$

For the variational inequality above, we consider the following regimes and associated free boundaries: 
\begin{eqnarray*}
&&\mathcal F:= \big\{(y,t)\in \mathcal Q ~ | ~ \hat{u}(y,t)=0\big\}{\mbox{ (Function\ constraint region)}}, \\
&&\mathcal G:= \big\{(y,t)\in \mathcal Q ~ | ~ (1-p)\hat{u}(y,t)+py\partial_y\hat{u}(y,t)=0\big\}{\mbox{ (Gradient\ constraint region)}}, \\
&&\mathcal E:= \big\{(y,t)\in \mathcal Q~ | ~  (1-p)\hat{u}(y,t)+py \partial_y\hat{u}(y,t)<0, \hat{u}(y,t)>0\big\}{\mbox{ (Equation region)}}.
\end{eqnarray*}

We plot in Figure \ref{fig-1} the numerical illustration of the above free boundaries as below: 
\begin{figure}[H]
    \centering
\includegraphics[width=7cm]{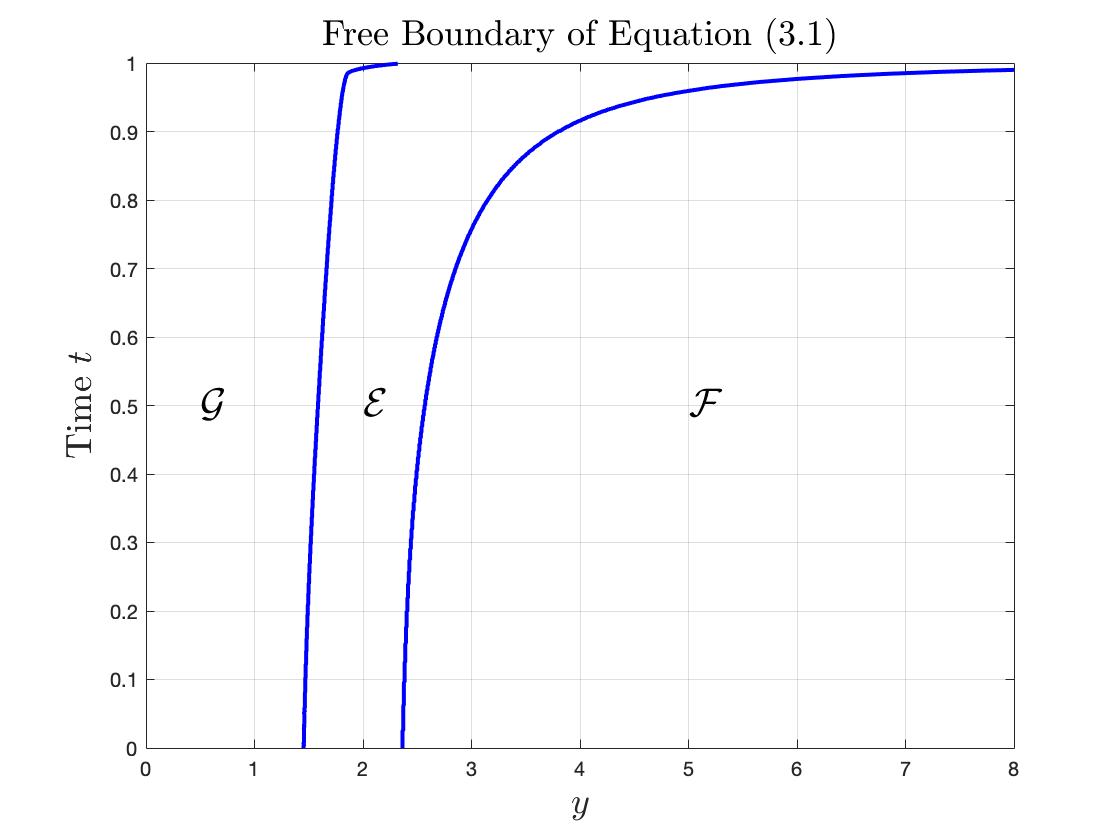}
\caption{The numerical illustration of free boundaries in the variational inequality $(3.1)$ }\label{fig-1}
\end{figure}

To simplify some analysis, let us further set $y=e^s$,$\tau=T-t$ and $\tilde u(s,\tau)=\hat u(y,t)$. It follows that $\tilde u(s,\tau)$ satisfies
\begin{equation}\label{eq:VItildeu}
\left\{\begin{array}{l}
\min\left\{\partial_\tau\tilde u-\frac 1 2\frac{\mu^2}{\sigma^2}\partial_{ss}\tilde u-(\delta-\frac 1 2 \frac{\mu^2}{\sigma^2})\partial_s\tilde u+\delta\tilde u+\tilde f(s),\ (p-1)\tilde u-p\partial_s\tilde u,\ \tilde u\right\}=0, \quad (s,\tau)\in\Omega,\\
\tilde u(s,0)=\frac{p}{1-p}e^{\frac{p-1}{p}s},\quad s\in\mathbb R,
\end{array}\right.
\end{equation}
where $\Omega=(-\infty,+\infty)\times (0,T]$ and
\begin{eqnarray*}
\tilde f(s)=\left\{\begin{array}{ll}
\alpha e^s-\frac{\alpha^{1-p}}{1-p}, & \mbox{if } s\geq -p\ln\alpha,\\
-\frac{p}{1-p}e^{\frac{p-1}{p}s}, & \mbox{if } 0<s< -p\ln\alpha,\\
e^s-\frac{1}{1-p}, & \mbox{if } s\leq0.
\end{array}\right.
\end{eqnarray*}

The transformed regimes and the associated free boundaries of the variational inequality \eqref{eq:VItildeu} are plotted in Figure \ref{fig-2} as below.

\begin{figure}[H]
    \centering
\includegraphics[width=7cm]{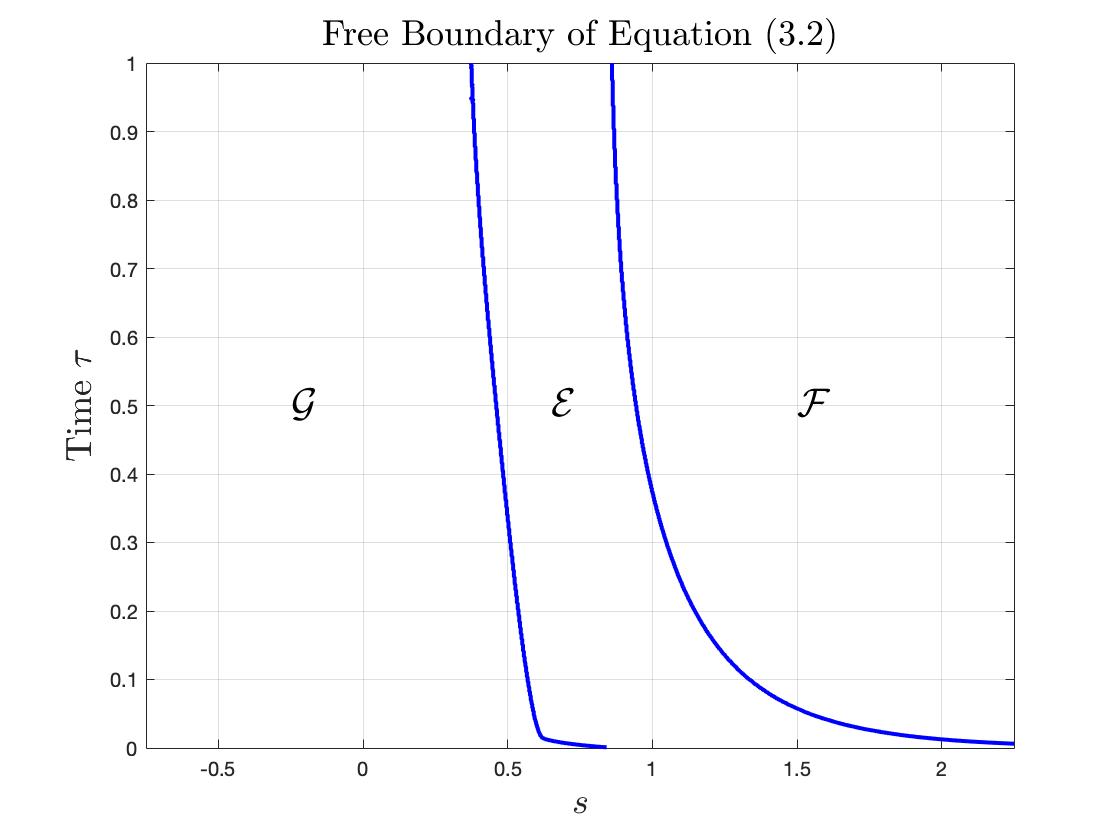}
\caption{The numerical illustration of free boundaries in the variational inequality $(3.2)$ }\label{fig-2}

\end{figure}

Considering
\begin{equation}\label{eq:vtildeu}
v(s,\tau):=e^{\frac{1-p}{p}s}\tilde u(s,\tau),
\end{equation}
%
%
we can work with another auxiliary dual variational inequality of $v(s,\tau)$ that
\begin{eqnarray}\label{eq:VIv}
\left\{\begin{array}{ll}
\min\left\{\partial_\tau v-\mathcal L_sv+e^{\frac{1-p}{p}s}\tilde f(s),-\partial_sv(s,\tau),v(s,\tau)\right\}= 0,\quad (s,\tau)\in\Omega,\\
v(s,0)=\frac{p}{1-p},\quad s\in\mathbb R,
\end{array}\right.
\end{eqnarray}
where 
$$
\mathcal L_sv=\frac 1 2\frac{\mu^2}{\sigma^2}\partial_{ss}v-\left(\frac{2-p}{2p}\frac{\mu^2}{\sigma^2}-\delta\right)\partial_{s}v-\left(\frac{1}{p}\delta-\frac{\mu^2(1-p)}{2p^2\sigma^2}\right) v.
$$

In what follows, we will first study the existence and uniqueness of the solution to the auxiliary problem \eqref{eq:VIv} and investigate its associated free boundary curve. Then, based on the transform \eqref{eq:vtildeu}, we can examine the auxiliary variational inequalities \eqref{eq:VItildeu} and \eqref{eq:VIhatu}. Finally, we can conclude the existence and uniqueness of the solution to the dual variational inequality \eqref{eq:VIu}.

\subsection{Characterization of the free boundary in \eqref{eq:VIv}}
Let us first analyze the auxiliary variational inequality \eqref{eq:VIv}, which is a parabolic variational inequality with both gradient constraint and function constraint.
We can obtain the existence and uniqueness of the solution in $W^{2,1}_{q,loc}(\Omega)\cap C(\overline{\Omega})$ to problem \eqref{eq:VIv} in the next result, where the Sobolev space $W^{2,1}_{q,loc}$ is defined by
 $$
 W^{2,1}_{q,loc}(\Omega) = \big\{v(s,\tau) |v, \partial_sv, \partial_{ss}v, \partial_\tau v\in \mathcal{L}^q(\Omega'),\;\Omega'\subset\subset\Omega\big\},\quad 1<q<\infty.
 $$

In order to obtain the existence and properties of the solution to problem \eqref{eq:VIv}, we first consider the following problem, for any $N>0$, $v_N$ satisfies
\begin{eqnarray}\label{eq:VIvN}
\left\{\begin{array}{ll}
\min\left\{\partial_\tau v_N-\mathcal L_sv_N+e^{\frac{1-p}{p}s}\tilde f(s),-\partial_sv_N(s,\tau),v_N(s,\tau)\right\}= 0,\quad (s,\tau)\in\Omega_N,\\
v_N(s,0)=\frac{p}{1-p},\quad s>-N,\\
v_N(-N,\tau)=\frac{p}{1-p},\quad \tau\in[0,T],
\end{array}\right.
\end{eqnarray}
where  $\Omega_N=(-N,+\infty)\times(0,T]$.

\begin{lemma}\label{lem:vn}
There exists a  solution $v_N(s,\tau)\in W^{2,1}_{q,loc}(\Omega_N)\cap C(\overline{\Omega_N})$ to problem \eqref{eq:VIvN}. Moreover, we have
\begin{eqnarray}
&&0\leq v_N(s,\tau)\leq \frac{p}{1-p},\quad (s,\tau)\in\Omega_N,\label{eq:0vn}\\
&&\partial_\tau v_N(s,\tau)\leq0,\quad (s,\tau)\in\Omega_N,\label{eq:ptauvn}\\
&&\lim\limits_{s\rightarrow+\infty}v_N(s,\tau)=0.\label{eq:vninfty},
\end{eqnarray}
and the solution satisfying \eqref{eq:0vn} is unique.
\end{lemma}
\begin{proof}
We can solve the problem \eqref{eq:VIvN} using the standard penalty approximation method. Consider $v_{N}^\varepsilon$ satisfies 
\begin{eqnarray}\label{eq:VIvNep}
\left\{\begin{array}{ll}
\partial_\tau v_{N}^\varepsilon-\mathcal L_sv_{N}^\varepsilon+e^{\frac{1-p}{p}s}\tilde f(s)+\beta_\varepsilon(v_{N}^\varepsilon)+s\beta_\varepsilon(-\partial_sv_{N}^\varepsilon)= 0,\quad (s,\tau)\in\Omega_N,\\
v_{N}^\varepsilon(s,0)=\frac{p}{1-p},\quad s>-N,\\
v_{N}^\varepsilon(-N,\tau)=\frac{p}{1-p},\quad \tau\in[0,T],
\end{array}\right.
\end{eqnarray}
where $\beta_\varepsilon(\lambda)$ is the penalty function satisfying
\begin{eqnarray*}
&&\beta _{\varepsilon }(\cdot)\in C^{2}(-\infty,+\infty ), \quad \beta _{\varepsilon 
}(\cdot)\leq 0, \quad \beta _{\varepsilon}^{\prime }(\cdot)\geq 0,\\
&&\lim\limits_{\varepsilon  \rightarrow 0} \beta _{\varepsilon }(\lambda)=\left\{
\begin{array}{ll}
0, & \lambda>0,\\
-\infty, & \lambda<0.
\end{array}
\right.
\end{eqnarray*}
Using the standard fixed point theorem, we are able to show that there exists a solution $v_{N}^\varepsilon(s,\tau)\in W^{2,1}_{q,loc}(\Omega_N)\cap C(\overline{\Omega_N})$ to the penalty problem \eqref{eq:VIvNep}. Letting $\varepsilon\rightarrow0$, we obtain a solution $v_{N}(s,\tau)\in W^{2,1}_{q,loc}(\Omega_N)\cap C(\overline{\Omega_N})$ to problem \eqref{eq:VIvN}.

The estimate \eqref{eq:0vn} follows from the facts $\partial_sv_N\leq0, v_N\geq0$ and the boundary condition $v_N(-N,\tau)=\frac{p}{1-p}$. Moreover, based on the boundary condition $v_N(-N,\tau)=\frac{p}{1-p}$, we have the uniqueness of the solution to problem \eqref{eq:VIvN} and then the comparison principle for problem \eqref{eq:VIvN} holds true.

Now we will show \eqref{eq:ptauvn}. For any $0<\delta<T$, set $v_\delta(s,\tau)=v_N(s,\tau+\delta)$, then
 $v_\delta(s,\tau)$ satisfies
\begin{eqnarray}\label{eq:VIvdel}
\left\{\begin{array}{ll}
\min\left\{\partial_\tau v_\delta-\mathcal L_sv_\delta+e^{\frac{1-p}{p}s}\tilde f(s),-\partial_sv_\delta,v_\delta\right\}= 0,\quad (s,\tau)\in(-N,+\infty)\times(0,T-\delta],\\
v_\delta(s,0)=v_N(s,\delta)\leq\frac{p}{1-p}=v_N(s,0),\quad s>-N, \\
v_\delta(-N,\tau)=\frac{p}{1-p}=v_N(-N,\tau),\quad \tau\in[0,T-\delta].
\end{array}\right.
\end{eqnarray}
Using the comparison principle between \eqref{eq:VIvN} and \eqref{eq:VIvdel}, we have that
$$v_N(s,\tau+\delta)=v_\delta(s,\tau)\leq v_N(s,\tau),$$
which leads to the desired result \eqref{eq:ptauvn}.

For each $\tau>0$, define $$Z_N(\tau)=\min\{s|v_N(s,\tau)=0\}.$$
We next show that $Z_N(\tau)$ is finite for $\tau\in(0,T]$. Suppose that there exists a $\tau_0>0$ such that $Z_N(\tau_0)=+\infty$. 
Together with \eqref{eq:ptauvn}, it holds that 
$$
v_N(s,\tau)>0,\quad (s,\tau)\in(-N,+\infty)\times(0,\tau_0), 
$$
which implies
\begin{eqnarray}\label{eq:VIvN0}
\left\{\begin{array}{ll}
\min\left\{\partial_\tau v_N-\mathcal L_sv_N+e^{\frac{1-p}{p}s}\tilde f(s),-\partial_sv_N(s,\tau)\right\}= 0,\quad (s,\tau)\in(-N,+\infty)\times(0,\tau_1),\\
v_N(s,0)=\frac{p}{1-p},\quad s>-N,\\
v_N(-N,\tau)=\frac{p}{1-p},\quad \tau\in[0,\tau_1],
\end{array}\right.
\end{eqnarray}
where $0<\tau_1\leq \tau_0$.

Note that
\begin{eqnarray*}
e^{\frac{1-p}{p}s}\tilde f(s)&=&\alpha e^{\frac1p s}-\frac{\alpha^{1-p}}{1-p}e^{\frac{1-p}{p}s}\\
&=&e^{\frac1p s}\left(\alpha -\frac{\alpha^{1-p}}{1-p}e^{-s}\right)\rightarrow +\infty,\quad {\rm as}\;s\rightarrow+\infty, 
\end{eqnarray*}
it follows that there exists a constant $s_0$ such that
$$e^{\frac{1-p}{p}s}\tilde f(s)\geq \frac{3\alpha}{4}e^{\frac1p s},\quad s\geq s_0.$$
Let $\tilde v(s,\tau)=\frac{p}{1-p}-\frac\alpha2(e^{\frac1p s}-e^{\frac1p s_0})\tau$. We next show that $\tilde v(s,\tau)$ is a super-solution to \eqref{eq:VIvN0} on $(s_0,+\infty)\times(0,\tau_1)$.

When $\tau_1$ is small enough, we can deduce
\begin{eqnarray*}
&&\partial_\tau \tilde v-\mathcal L_s\tilde v+e^{\frac{1-p}{p}s}\tilde f(s)\\
&=&-\frac\alpha2\left(e^{\frac1p s}-e^{\frac1p s_0}\right)+\left(\frac\delta p-\frac{\mu^2(1-p)}{2p\sigma^2}\right)\left(\frac{p}{1-p}+\frac\alpha 2e^{\frac1p s_0}\tau\right)+e^{\frac{1-p}{p}s}\tilde f(s)\geq0.
\end{eqnarray*}
Together with $v_N(s,\tau)\leq\frac{p}{1-p}$, we know that $\tilde v(s,\tau)$ satisfies
\begin{eqnarray*}
\left\{\begin{array}{ll}
\min\left\{\partial_\tau \tilde v-\mathcal L_s\tilde v+e^{\frac{1-p}{p}s}\tilde f(s),-\partial_s\tilde v(s,\tau)\right\}= 0,\quad (s,\tau)\in(s_0,+\infty)\times(0,\tau_1),\\
\tilde v(s,0)=\frac{p}{1-p},\quad s>s_0,\\
\tilde v(-N,\tau)=\frac{p}{1-p}\geq v_N(s_0,\tau),\quad \tau\in[0,\tau_1].
\end{array}\right.
\end{eqnarray*}
The comparison principle implies that
$$v_N(s,\tau)\leq  \tilde v(s,\tau),\quad (s,\tau)\in(s_0,+\infty)\times(0,\tau_1].$$
Moreover, we have
$$\tilde v(s,\tau)<0,\quad s>p\ln \left(e^{\frac1ps_0}+\frac2\alpha\frac{p}{(1-p)\tau_1}\right),$$
which leads to
$$ v_N(s,\tau)<0,\quad s>p\ln \left(e^{\frac1ps_0}+\frac2\alpha\frac{p}{(1-p)\tau_1}\right),$$
which is a contradiction. As a result, $Z_n(\tau)$ is finite for each $\tau>0$, and it implies that \eqref{eq:vninfty} holds.
\end{proof}

\begin{proposition}\label{thm:VIv}
There exists a  solution $v(s,\tau) \in W^{2,1}_{q,loc}(\Omega)\cap C(\overline{\Omega})$ to problem \eqref{eq:VIv}. Moreover, we have 
\begin{eqnarray}
&&0\leq v(s,\tau)\leq \frac{p}{1-p},\quad (s,\tau)\in\Omega,\label{eq:0v}\\
&&\partial_\tau v(s,\tau)\leq0,\quad (s,\tau)\in\Omega,\label{ie:ptauv}\\
&&\lim\limits_{s\rightarrow+\infty}v(s,\tau)=0.\label{eq:infinity}
\end{eqnarray}
The solution to problem \eqref{eq:VIv} satisfying \eqref{eq:0v} and \eqref{eq:infinity}  is unique.
\end{proposition}

\begin{proof}
Denote $v(s,\tau)=\lim\limits_{N\rightarrow+\infty}v_N(s,\tau)$, then $v(s,\tau) \in W^{2,1}_{q,loc}(\Omega)\cap C(\overline{\Omega})$ is the solution of problem \eqref{eq:VIv}. It is easy to see that \eqref{eq:0v} and \eqref{ie:ptauv} can be derived from \eqref{eq:0vn} and \eqref{eq:ptauvn}.

We next show \eqref{eq:infinity}. Suppose $N_1>N_2$, then
$$v_{N_1}(-N_2,\tau)\leq \frac{p}{1-p}=v_{N_2}(-N_2,\tau).$$
By the comparison principle, we have
$$v_{N_1}(s,\tau)\leq v_{N_2}(s,\tau),\quad (s,\tau)\in(-N_2,+\infty)\times(0,T].$$
It hence holds that
$$v_{N_1}(s,\tau)=0,\quad (s,\tau)\in(Z_{N_2},+\infty)\times(0,T],$$
which implies that $Z_N(\tau)$ is decreasing in $N$. This, together with \eqref{eq:vninfty}, implies \eqref{eq:infinity}. 

We are ready to show the uniqueness of the solution to problem \eqref{eq:VIv} by contradiction. Suppose $v_1, v_2$ are two solutions to the problem \eqref{eq:VIv}. Denote $\mathcal N=\{v_1>v_2\}\neq\emptyset$, and let $\tau_0=\inf\{\tau:(s,\tau)\in\mathcal N\},\;A=\partial_p\mathcal N\cap \{(s,\tau):\tau=\tau_0\}$. It then holds that  $$v_1=v_2,\quad{\rm on}\;A.$$
Denote $\mathcal N_1=\mathcal N\cap\{(s,\tau): \partial_sv_1<\partial_sv_2\}$ and $\mathcal N_2=\mathcal N\cap\{(s,\tau): \partial_sv_1\geq\partial_sv_2\}$. Suppose $\mathcal N_1\neq\emptyset$, by condition \eqref{eq:infinity}, it is easy to show that there exists a point $(s_0,\tau_0)\in\partial_p \mathcal N_1\cap A$, which implies $v_1(s_0,\tau_0)=v_2(s_0,\tau_0)$, and we have 
\begin{eqnarray*}
\left\{\begin{array}{ll}
\partial_\tau v_1-\mathcal L_sv_1+e^{\frac{1-p}{p}s}\tilde f(s)=0,\quad (s,\tau)\in\mathcal N_1,\\
\partial_\tau v_2-\mathcal L_sv_2+e^{\frac{1-p}{p}s}\tilde f(s)\geq0,\quad (s,\tau)\in\mathcal N_1,\\
v_1=v_2\ {\rm or}\; \partial_sv_1=\partial_sv_2,\quad (s,\tau)\in\partial_p\mathcal N_1.\\
\end{array}\right.
\end{eqnarray*}
By condition \eqref{eq:0v} and the maximum principle, we know
$$v_2(s,\tau)\geq v_1(s,\tau),\quad (s,\tau)\in\mathcal N_1,$$
which contradicts the definition of $\mathcal N$ and hence $\mathcal N\subset\mathcal N_2$. 

We then conclude that
\begin{eqnarray*}
\left\{\begin{array}{ll}
\partial_sv_1\geq\partial_sv_2,\quad (s,\tau)\in\mathcal N,\\
v_1=v_2,\quad (s,\tau)\in\partial_p\mathcal N.\\
\end{array}\right.
\end{eqnarray*}
This, together with the condition \eqref{eq:infinity}, implies that 
 $$v_1(s,\tau)\leq v_2(s,\tau),\quad  (s,\tau)\in\mathcal N,$$
which is a contradiction to the definition of $\mathcal N$. The proof is then complete.
\end{proof}

As a direct result of the uniqueness of solution to problem \eqref{eq:VIv}, the comparison principle for problem \eqref{eq:VIv} holds.

Recall the transformation in \eqref{eq:vtildeu}, we have the following regions (see the numerical illustrations in Figure  \ref{fig-2}) that
\begin{eqnarray*}
&&\mathcal F= \big\{(s,\tau)\in \Omega ~ | ~ v(s,\tau)=0\big\}{\mbox{ (Function\ constraint region)}}, \\
&&\mathcal G= \big\{(s,\tau)\in \Omega ~ | ~ \partial_sv(s,\tau)=0\big\}{\mbox{ (Gradient\ constraint region)}}, \\
&&\mathcal E= \big\{(s,\tau)\in \Omega ~ | ~ \partial_sv(s,\tau)<0, v(s,\tau)>0\big\}{\mbox{ (Equation region)}}.
\end{eqnarray*}
%


Note that $v(s,\tau)$ satisfies $\partial_sv(s,\tau)\leq 0$ and $v(s,\tau)\geq 0$. For each $\tau>0$, let us define
\begin{equation}\label{eq:ztau}
Z(\tau):=\inf\big\{s ~ | ~ v(s,\tau) = 0\big\}.
\end{equation}

\begin{proposition}\label{thm:mf}
The curve $Z(\tau)$ defined in \eqref{eq:ztau} satisfies $-p\ln\alpha<Z(\tau)<+\infty$ for $\tau\in(0,T]$ and
\begin{equation}\label{eq:fbz}
\mathcal F= \big\{(s,\tau)\in \mathcal Q ~ | ~ s\geq Z(\tau)\big\}.
\end{equation}
Moreover, $Z(\tau)$ strictly decreases in $\tau$. In particular,
\begin{equation}\label{lim:z0}
\lim\limits_{\tau\rightarrow 0^+} Z(\tau)=+\infty.
\end{equation}
\end{proposition}
\begin{proof}
The result \eqref{eq:fbz} follows from definitions of $Z(\tau)$ and $\mathcal{F}$. By the variational inequality \eqref{eq:VIv}, we have
$$\partial_\tau v-\mathcal L_sv+e^{\frac{1-p}{p}s}\tilde f(s)\geq0,\quad\text{if}\ v=0,$$
which leads to $s\geq-p\ln\alpha-\ln(1-p)$. Hence, by the definition of $Z(\tau)$, we know $Z(\tau)>-p\ln\alpha$.

Next, we show that $Z(\tau)$ is finite for $\tau\in(0,T]$. Suppose that there exists a $\tau_0>0$ such that $Z(\tau_0)=+\infty$. It then holds that
$$v(s,\tau)>0, \quad (s,\tau)\in\mathbb R\times(0,\tau_0],$$
which implies that $v(s,\tau)$ satisfies
\begin{equation}\label{eq:VIvs0}
\left\{\begin{array}{ll}
\min\left\{\partial_\tau v-\mathcal L_sv+e^{\frac{1-p}{p}s}\tilde f(s),-\partial_sv(s,\tau)\right\}= 0,\quad (s,\tau)\in\mathbb R\times(0,\tau_0],\\
v(s,0)=\frac{p}{1-p},\quad s\in\mathbb R.
\end{array}\right.
\end{equation}
Note that
$$e^{\frac{1-p}{p}s}\tilde f(s)=\alpha e^{\frac 1p s}-\frac{\alpha^{1-p}}{1-p}e^{\frac{1-p}{p}s}\rightarrow +\infty\ \text{as}\ s\rightarrow+\infty.$$
There exists a constant $s_0$ such that
$$e^{\frac{1-p}{p}s}\tilde f(s)\geq \frac{\alpha}{2}e^{\frac 1p s},\quad s\geq s_0.$$
Let $\tilde v(s,\tau):=\frac{p}{1-p}-\frac{\alpha}{2}(e^{\frac 1p s}-e^{\frac 1p s_0})\tau$. We next show that $\tilde v(s,\tau)$ is a super-solution to \eqref{eq:VIvs0} on $(s_0,+\infty)\times(0,\tau_0)$. In view that
\begin{eqnarray*}
&&\partial_\tau \tilde v-\mathcal L_s\tilde v+e^{\frac{1-p}{p}s}\tilde f(s) \\
&=&-\frac{\alpha}{2}(e^{\frac 1p s}-e^{\frac{1}{p}s_0})+\left(\frac{\delta}{p}-\frac{\mu^2(1-p)}{2p^2\sigma^2}\right)\left(\frac{p}{1-p}+\frac{\alpha}{2}e^{\frac{1}{p}s_0}\tau\right)+e^{\frac{1-p}{p}s}\tilde f(s)\geq0. 
\end{eqnarray*}
Together with $\frac{p}{1-p}\geq v(s,\tau)$, we know that $\tilde v(s,\tau)$ satisfies
\begin{eqnarray*}
\left\{\begin{array}{ll}
\min\left\{\partial_\tau \tilde v-\mathcal L_s\tilde v+e^{\frac{1-p}{p}s}\tilde f(s),-\partial_s\tilde v(s,\tau)\right\}\geq 0,\quad & (s,\tau)\in (s_0,+\infty)\times(0,\tau_0],\\
\tilde v(s,0)=\frac{p}{1-p},& s\in (s_0,+\infty),\\
\tilde v(s_0,\tau)=\frac{p}{1-p}\geq v(s_0,\tau), & \tau\in(0,\tau_0].
\end{array}\right.
\end{eqnarray*}
The comparison principle implies that
$$v(s,\tau)\leq \tilde v(s,\tau),\quad (s,\tau)\in(s_0,+\infty)\times(0,\tau_0].$$
Moreover, it is easy to see that
$$\tilde v(s,\tau)<0,\quad s>p\ln\Big(e^{\frac 1ps_0}+\frac{2}{\alpha}\frac{p}{(1-p)\tau_0}\Big),$$
which implies that
$$v(s,\tau)<0,\quad s>p\ln\Big(e^{\frac 1ps_0}+\frac{2}{\alpha}\frac{p}{(1-p)\tau_0}\Big),$$
leading to a contradiction.

Finally, we show the strict monotonicity of $Z(\tau)$. For any $\tau_0>0$, suppose $Z(\tau_0)=s_0$, i.e. $v(s_0,\tau_0)=0.$
According to $\partial_\tau v(s,\tau)\leq0$, we have
$v(s_0,\tau)=0$, $\tau\in(\tau_0,T]$. By the definition \eqref{eq:ztau} of  $Z(\tau)$, we obtain
$$Z(\tau)\leq s_0,\quad \tau\in(\tau_0,T].$$
Hence, we obtain the monotonicity of $Z(\tau)$. Suppose that $Z(\tau)$ is not strictly monotone and there exists $\tau_1<\tau_2$, such that$$Z(\tau)=s_0,\quad \tau\in[\tau_1,\tau_2].$$
Denote $\Gamma:=\{s_0\}\times(\tau_1,\tau_2)$. Then we have
\begin{eqnarray*}
\left\{\begin{array}{ll}
\partial_\tau v-\mathcal L_s v=-e^{\frac{1-p}{p}s}\tilde f(s),\quad (s,\tau)\in (s_0-\varepsilon,s_0)\times(\tau_1,\tau_2), \\
v|_{\Gamma}=\partial_sv|_{\Gamma}=0,
\end{array}\right.
\end{eqnarray*}
where $\varepsilon$ is small enough. It then follows that
\begin{eqnarray*}
\left\{\begin{array}{ll}
\partial_\tau (\partial_\tau v)-\mathcal L_s  (\partial_\tau v)=0,\quad (s,\tau)\in (s_0-\varepsilon,s_0)\times(\tau_1,\tau_2), \\
\partial_\tau v|_{\Gamma}=\partial_{s\tau}v|_{\Gamma}=0.
\end{array}\right.
\end{eqnarray*}
Together with the fact that $\partial_\tau v\leq0$, Hopf's principle implies that
$$\partial_{s\tau}v|_{\Gamma}>0 \quad\mbox{ or }\quad \partial_{\tau}v\equiv 0,\quad (s,\tau)\in (s_0,s_0+\varepsilon)\times(\tau_1,\tau_2),$$
leading to a contradiction.

Moreover, it follows from the monotonicity of $Z(\tau)$ and $v(s,0)=\frac{p}{1-p}>0$ that the result \eqref{lim:z0} holds.
\end{proof}

Let us next focus on the domain $\{s\leq Z(\tau)\}$. In view of the definition of $Z(\tau)$, $v(s,\tau)$ is a unique $W^{2,1}_{q,loc}(\Omega)$ solution to problem \eqref{eq:VIv}.
On the domain $\{s\leq Z(\tau)\}$, $v(s,\tau)$ satisfies
\begin{equation}\label{eq:VIv1}
\left\{\begin{array}{ll}
\min\left\{\partial_\tau v-\mathcal L_sv+e^{\frac{1-p}{p}s}\tilde f(s),-\partial_sv(s,\tau)\right\}= 0, & (s,\tau)\in\tilde\Omega, \\
v(s,0)=\frac{p}{1-p}, & s\in\mathbb R, \\
\partial_s v(Z(\tau),\tau)=0, & \tau\in(0,T],
\end{array}\right.
\end{equation}
where $\tilde\Omega:=(-\infty,Z(\tau))\times(0,T]$. 
In order to analyze the free boundary arising from gradient constraint,
we follow the similar idea in \cite{CY12} and consider the parabolic obstacle problem
\begin{equation}\label{eq:VIw}
\left\{\begin{array}{ll}
\max\big\{\partial_\tau w(s,\tau)-\mathcal L_sw(s,\tau)-g(s), w(s,\tau)\big\}=0,\quad & (s,\tau)\in\tilde\Omega,\\
w(s,0)=0,& s\in\mathbb R,\\
w(Z(\tau),\tau)=0,&\tau\in(0,T],
\end{array}\right.
\end{equation}
where
\begin{equation}
g(s)=\Big(-e^{\frac{1-p}{p}s}\tilde f(s)\Big)'
=\frac{\alpha}{p}e^{\frac{1-p}{p}s}(\alpha^{-p}-e^s)I_{\{e^s\geq\alpha^{-p}\}}+\frac 1 pe^{\frac{1-p}{p}s}(1-e^{s})I_{\{e^s\leq1\}}.\label{eq:g}
\end{equation}
%

 \subsection{Characterization of the free boundary in problem \eqref{eq:VIw}}

 Following the standard penalty approximation method as to show the existence of solution to the problem \eqref{eq:VIv},
  it is easy to conclude the next result, and its proof is hence omitted.
 

\begin{lemma} \label{thm:w}
There exists a unique $w(s,\tau)\in W^{2,1}_{q,loc}(\tilde\Omega)\cap C(\overline{\tilde\Omega})$ to problem \eqref{eq:VIw} for any $1<q<+\infty$.
\end{lemma}



To study some properties of the free boundary in \eqref{eq:VIw}, let us first define
\begin{eqnarray*}
&&\mathcal{G}_0:= \big\{(s,\tau)\in\tilde\Omega ~ | ~ w(s,\tau)=0\big\}, \\
&&\mathcal{E}_0 := \big\{(s,\tau)\in \tilde\Omega ~ | ~ w(s,\tau)<0\big\}.
\end{eqnarray*}

\begin{proposition}\label{thm:Ts}
There exists a function $T(s):(-\infty,+\infty)\rightarrow [0,T]$ such that
\begin{equation}
\mathcal{G}_0 = \big\{(s,\tau) ~ | ~ s\in\mathbb R,0\leq\tau\leq T(s)\big\},\label{eq:fG}
\end{equation}
and $T(s)$ is decreasing in $s$ such that
\begin{eqnarray}
T(s)=0,\quad s>0,\label{eq:T0}\\
T(s)>0,\quad s<0.\label{ie:T0}
\end{eqnarray}
Particularly, $T(s)$ is strictly decreasing on $\{s ~ | ~ 0<T(s)<T\}$ and $T(s)$ is continuous.
\end{proposition}
\begin{proof}
The conjectured free boundary $\omega^*(t)$ in Remark \ref{rmk:omega} and all the transformations in Section 2 imply that $\mathcal{G}_0$ is connected in $\tau$ direction. Note that $w(s,0)=0$, let us define
$$T(s):=\sup\big\{\tau ~ | ~ w(s,\tau)=0\big\},\quad \forall s\in\mathbb{R}.$$
By the definitions of $\mathcal{G}_0$ and $T(s)$, we get the desired result \eqref{eq:fG}.

We next show the monotonicity of $T(s)$. By the variational inequality \eqref{eq:VIw}, we have
$$g(s)\geq0,\quad{\rm if}\;w(s,\tau)=0,$$
which implies $s\leq -p\ln\alpha$. That is, $w(s,\tau)<0$ for $s> -p\ln\alpha$, $\tau\in(0,T]$. It follows that 
$$\big\{(s,\tau) ~ | ~ s> -p\ln\alpha, ~ \tau\in(0,T]\big\}\subset\mathcal{E}_0.$$
For any $s_0\leq-p\ln\alpha$ such that $T(s_0)>0$, we define an auxiliary function
\begin{eqnarray*}
\tilde w(s,\tau):=\left\{\begin{array}{ll}
0, & \mbox{if } (s,\tau)\in(-\infty,s_0]\times[0,T(s_0)],\\
w(s,\tau), & \mbox{if } (s,\tau)\in\{(s_0,+\infty)\times[0,T(s_0)]\}\cap\tilde\Omega.
\end{array}\right.
\end{eqnarray*}
We show that $\tilde w(s,\tau)$ is the solution to problem \eqref{eq:VIw} in the domain $\{\mathbb R\times [0,T(s_0)]\}$.
By the definition of $\tilde w(s,\tau)$, we have that $\tilde w(s,0)=0$, $\tilde w(s,\tau)\leq 0$, and
\begin{eqnarray*}
\left\{\begin{array}{ll}
\partial_\tau\tilde  w-\mathcal L_s\tilde w=0\leq g(s), & \mbox{if } (s,\tau)\in(-\infty,s_0]\times[0,T(s_0)],\\
\partial_\tau\tilde w-\mathcal L_s\tilde w=\partial_\tau w-\mathcal L_s w\leq g(s), & \mbox{if } (s,\tau)\in\{(s_0,+\infty)\times[0,T(s_0)]\}\cap\tilde\Omega.
\end{array}\right.
\end{eqnarray*}
Moreover, if $\tilde w(s,\tau)<0$, then $w(s,\tau)<0$. Hence, we have
$$\partial_\tau\tilde w-\mathcal L_s\tilde w=\partial_\tau w-\mathcal L_s w= g(s),\quad \tilde w(s,\tau)<0.$$
Thus, $\tilde w(s,\tau)$ is a $W^{2,1}_{q,loc}$-solution to problem \eqref{eq:VIw} in the domain $\{\mathbb R\times [0,T(s_0)]\}\cap\tilde\Omega$. The uniqueness of the solution to \eqref{eq:VIw} yields that
$$w(s,\tau)=\tilde w(s,\tau)=0,\quad (s,\tau)\in(-\infty,s_0]\times[0,T(s_0)].$$
By the definition of $T(s)$, we obtain $T(s)\geq T(s_0),\quad \forall s<s_0$, and $T(s)$ is decreasing in $s$.

Next, we show \eqref{eq:T0}. Suppose that there exists $s_0>0$ such that $T(s_0)=\tau_0>0$, by the monotonicity of $T(s)$, we have
\begin{eqnarray*}
\left\{\begin{array}{ll}
\partial_\tau w-\mathcal L_sw\leq g(s)\leq0,\quad &(s,\tau)\in\{[0,+\infty)\times [0,\tau_0]\}\cap\tilde\Omega, \\
w(0,\tau)=0, \quad w(Z(\tau),\tau)=0, & \tau\in(0,\tau_0], \\
w(s,0)=0, & s\in\mathbb R^+.
\end{array}\right.
\end{eqnarray*}
The strong maximum principle implies that
$$w(s,\tau)<0,\quad (s,\tau)\in(0,+\infty)\times (0,\tau_0].$$
It contradicts with $w(s_0,\tau)=0,\quad \tau\in(0,\tau_0].$ Hence, \eqref{eq:T0} holds true.

In view of the definition of $T(s)$ and the fact $w(s,\tau)\leq0$, we have
\begin{eqnarray}\label{ie:ptw}
\partial_\tau w(s,0)\leq 0,\quad \forall s\in \mathbb R,\label{ie:ptw0} \\
\partial_\tau w(s,T(s))\leq 0,\quad \forall s\in \mathbb R.\label{ie:ptw}
\end{eqnarray}
%
Suppose that there exists $s_1<0$ such that $T(s_1)=0$, then we have
$$\partial_\tau w(s,0)=\mathcal L_s w(s,0)+g(s)>0,\quad s\in(s_1,0),$$
yielding a contraction to \eqref{ie:ptw0}. Hence \eqref{ie:T0} follows.

Thanks to \eqref{ie:ptw},  we claim the strict monotonicity of $T(s)$ in $\big\{s ~ | ~ 0<T(s)<T\big\}$. Indeed, suppose that exists $s_1<s_2\leq0$ such that
$$0<T(s_1)=T(s_2)<T.$$
Then we have
$$w(s,T(s_2))=0,\quad s\in(s_1,s_2).$$
Applying the equation $\partial_\tau w-\mathcal L_s w=g(s)$ at $(s_1,s_2)\times\{T(s_2)\}$, we have
$$\partial_\tau w|_{\tau=T(s_2)}=(\mathcal L_s w+g(s))|_{\tau=T(s_2)}>0,\quad s\in(s_1,s_2),$$
which contradicts with \eqref{ie:ptw}. The claim therefore holds.

Following the proof of the strictly monotonicity of $Z(\tau)$ in Theorem \ref{thm:Ts}, we can conclude the continuity of $T(s)$.
\end{proof}

As $T(s)$ strictly decreases and is continuous in $s$ on $0<T(s)<T$, there exists an inverse function of $T(s)$ denoted by $S(\tau):=T^{-1}(\tau)$, $\tau\in(0,T)$.  
Let us define
\begin{eqnarray}\label{freeS}
S(\tau):=\left\{\begin{array}{ll}
T^{-1}(\tau), & \mbox{if } \tau\in(0,T), \\
\sup\big\{s ~ | ~ T(s)=T\big\}, & \mbox{if } \tau=T.
\end{array}\right.
\end{eqnarray}
%

\begin{lemma} \label{thm:Sb} 
We have that
\begin{equation}
\mathcal{G}_0=\big\{(s,\tau) ~ | ~ s\leq S(\tau),\tau\in(0,T]\big\},\label{eq:fS}
\end{equation}
where $S(\tau)$ is continuous and strictly decreases in $\tau$ with
\begin{equation}
S(0)=\lim_{\tau\rightarrow0^+} S(\tau)=0.\label{eq:S0}
\end{equation}
\end{lemma}
\begin{proof}
%
Note that $T(s)$ strictly decreases in $\{s ~ | ~ 0<T(s)<T\}$, we know $S(\tau)$ decreases in $\tau$.
Because the strict monotonicity of $T(s)$ is equivalent to the continuity of $S(\tau)$ and the continuity of $T(s)$ is equivalent to the strict monotonicity of $S(\tau)$, we conclude that $S(\tau)\in C[0,T]$ and strictly decreases in $\tau$.

In view of the strict monotonicity of $S(\tau)$, let us define $S(0):=\lim_{\tau\rightarrow0^+} S(\tau)$. It follows from \eqref{eq:T0}-\eqref{ie:T0} that $S(0)=0$, which completes the proof.
\end{proof}

We next establish the dependence of $S(\tau)$ on the parameter $\alpha$ in the following result.

\begin{lemma} \label{thm:fbde}
The free boundary $S(\tau)$ of problem \eqref{eq:VIw} is decreasing in $\alpha$.
\end{lemma}
\begin{proof}
By the definition of $g(s)$, we have
\begin{eqnarray*}
\frac{\partial g}{\partial\alpha}=\frac1pe^{\frac{1-p}{p}s}(\alpha^{-p}-e^s-p\alpha^{-p})I_{\{e^s\geq\alpha^{-p}\}}<0.
\end{eqnarray*}
Let us choose $0<\alpha_1<\alpha_2\leq 1$. It holds that
\begin{eqnarray*}
g_1(s)&\overset{\triangle}{=}&\frac{\alpha_1}{p}e^{\frac{1-p}{p}s}(\alpha_1^{-p}-e^s)I_{\{e^s\geq\alpha_1^{-p}\}}+\frac 1 pe^{\frac{1-p}{p}s}(1-e^{s})I_{\{e^s\leq1\}}\\
&>&\frac{\alpha_2}{p}e^{\frac{1-p}{p}s}(\alpha_2^{-p}-e^s)I_{\{e^s\geq\alpha_2^{-p}\}}+\frac 1 pe^{\frac{1-p}{p}s}(1-e^{s})I_{\{e^s\leq1\}}\overset{\triangle}{=}g_2(s).
\end{eqnarray*}
For $i=1,2$, denote $w_i(s,\tau)$ as the solution to the following problem
\begin{equation}\label{eq:VIwi}
\left\{\begin{array}{ll}
\max\big\{\partial_\tau w_i(s,\tau)-\mathcal L_sw_i(s,\tau)-g_i(s), w_i(s,\tau)\big\}=0, & (s,\tau)\in\tilde\Omega, \\
w_i(s,0)=0, & s\in\mathbb{R}, \\
w_i(Z(\tau),\tau)=0, & \tau\in(0,T]. 
\end{array}\right.
\end{equation}
The comparison principle implies that $w_2(s,\tau)\leq w_1(s,\tau)$, for $(s,\tau)\in\Omega$. In particular, we have that
$$w_2(s,\tau)\leq w_1(s,\tau)<0,\quad s>S_1(\tau), \tau\in(0,T],$$
where $S_i(\tau)$ is the free boundary of problem \eqref{eq:VIwi}, $i=1,2$. According to \eqref{eq:fS}, we have that $S_2(\tau)\leq S_1(\tau)$, i.e., $S(\tau)$ decreases in $\alpha$.
\end{proof}

\subsection{The solution to problem \eqref{eq:VItildeu}}\label{sec:sol}
In this subsection, we first use the solution to problem \eqref{eq:VIw} to construct the solution to problem \eqref{eq:VIv1}, and then obtain the solution to problem \eqref{eq:VIv}.
Using the transform \eqref{eq:vtildeu} between $\tilde u(s,\tau)$ and $v(s,\tau)$, we can further obtain the solution to problem \eqref{eq:VItildeu}. Following the same proof of Theorem 4.6 in \cite{CY12}, we can get the next result.

\begin{proposition}\label{thm:v}
Let $w(s,\tau)$ be the solution to problem \eqref{eq:VIw} and let us define
\begin{equation}\label{eq:solv}
\bar v(s,\tau):=\int_{Z(\tau)}^sw(\xi,\tau)d\xi+\frac{p}{1-p}\chi\{\tau=0\},\quad (s,\tau)\in\tilde\Omega.
\end{equation}
%
%
%
Then $\bar v(s,\tau)$ is the unique solution to problem \eqref{eq:VIv1} satisfying
\begin{eqnarray}
&&\bar v(s,\tau)\in C^{2,1}(\tilde\Omega)\cap C(\overline{\tilde\Omega}),\label{reg:v}\\
&&\partial_s\bar v(s,\tau)\in W^{2,1}_{q,loc}(\tilde\Omega)\cap C(\overline{\tilde\Omega})\label{reg:psv}.
\end{eqnarray}
Moreover, if we define
\begin{equation}\label{def:v}
v(s,\tau)=\left\{
\begin{array}{ll}
\bar v(s,\tau), & \mbox{if } s\leq Z(\tau),\\
0, & \mbox{if } s>Z(\tau),
\end{array}\right.
\end{equation}
then $v(s,\tau)\in W^{2,1}_{q,loc}(\Omega)\cap C(\overline{\Omega})$ is the solution to problem \eqref{eq:VIv}. In addition, let $S(\tau)$ be given in \eqref{freeS} and let
$Z(\tau)$ be given in \eqref{eq:ztau}, $S(\tau)$ and $Z(\tau)$ are free boundaries of problem \eqref{eq:VIv} such that
\begin{eqnarray}
&&\mathcal{F} = \big\{v=0\}=\{(s,\tau) ~ | ~ s\geq Z(\tau),\tau\in(0,T]\big\},\label{eq:fv0}\\
&&\mathcal{G} = \big\{-\partial_sv=0,v>0\}=\{(s,\tau) ~ | ~ s\leq S(\tau),\tau\in(0,T]\big\},\label{eq:fv1}\\
&&\mathcal{E} = \big\{-\partial_sv>0,v>0\}=\{(s,\tau) ~ | ~ S(\tau)<s<Z(\tau),\tau\in(0,T]\big\}.\label{eq:fv2}
\end{eqnarray}
%
\end{proposition}

We also have the next result

\begin{lemma}\label{thm:pssv}
The solution $v(s,\tau)$ to problem \eqref{eq:VIv} satisfies
\begin{eqnarray}
&&\partial_{s}v(s,\tau)-\frac{1-p}{p}v(s,\tau)<0,\quad (s,\tau)\in\tilde\Omega,\label{ie:psv}\\
&&\partial_{ss}v(s,\tau)-\frac{2-p}{p}\partial_{s}v(s,\tau)+\frac{1-p}{p^2}v(s,\tau)>0,\quad (s,\tau)\in\tilde\Omega.\label{ie:pssv}
\end{eqnarray}
\end{lemma}

 In view of $S(\tau)\leq0$ and the definition of $\phi(\tau)$, $v(s,\tau)$ is continuous across the free boundary $Z(\tau)$. By the strong maximum principle, it is easy to show inequalities \eqref{ie:psv}-\eqref{ie:pssv} in Lemma \ref{thm:pssv} using standard arguments, and its proof is omitted.

Based on the relationship \eqref{eq:vtildeu}, it is straightforward to see that $\tilde u(s,\tau)=e^{\frac{p-1}{p}s}v(s,\tau)$ is a $W^{2,1}_{q,loc}(\Omega)\cap C(\overline \Omega)$-solution to problem \eqref{eq:VItildeu}. Hence, we have the following theorem.

\begin{proposition}\label{thm:tildeu}
For $(s,\tau)\in\tilde\Omega$, $\tilde u(s,\tau)=e^{\frac{p-1}{p}s}v(s,\tau)$ is the unique solution to problem \eqref{eq:VItildeu}, where $v(s,\tau)$ is the solution to problem \eqref{eq:VIv}.
In addition, $S(\tau)$ defined in \eqref{freeS} and $Z(\tau)$ defined in \eqref{eq:ztau} are free boundaries of problem \eqref{eq:VItildeu} such that
\begin{eqnarray}
&&\mathcal F=\{\tilde u=0\}=\{(s,\tau) ~ | ~ s\geq Z(\tau),\tau\in(0,T] \},\label{fb:tildeu0}\\
&& \mathcal G=\{(p-1)\tilde u-p\partial_s\tilde u=0,\tilde u>0\}=\{(s,\tau) ~ | ~ s\leq S(\tau)\},\label{fb:tildeu1}\\
&& \mathcal E=\{(p-1)\tilde u-p\partial_s\tilde u>0,\tilde u>0\}=\{(s,\tau) ~ | ~ S(\tau)<s<Z(\tau)\}.\label{fb:tildeu2}
\end{eqnarray}
Moreover, $\tilde u(s,\tau)\in W^{2,1}_{q,loc}(\Omega)\cap C(\overline \Omega)$ and $\tilde u(s,\tau)\in C^{2,1}(\tilde\Omega)\cap C(\overline{\tilde\Omega})$ that satisfies
\begin{eqnarray}
&&\partial_s\tilde u(s,\tau)<0,\quad (s,\tau)\in\tilde\Omega\label{ie:pstildeu}\\
&&\partial_{ss}\tilde u(s,\tau)-\partial_s\tilde u(s,\tau)>0, \quad (s,\tau)\in\tilde\Omega.\label{ie:psstildeu}
\end{eqnarray}
\end{proposition}
\begin{proof}
It follows from transform \eqref{eq:vtildeu} that $\tilde u(s,\tau)$ is the unique solution to \eqref{eq:VItildeu}.
The regularity of $\tilde u(s,\tau)$ can be deduced from the regularity of $v(s,\tau)$. 
The uniqueness of solution to problem \eqref{eq:VIv} leads to the uniqueness of solution to the problem \eqref{eq:VItildeu}. 
Using the transform \eqref{eq:vtildeu},
we can deduce \eqref{fb:tildeu0}-\eqref{fb:tildeu2} from \eqref{eq:fv0}-\eqref{eq:fv2}, and \eqref{ie:pstildeu}-\eqref{ie:psstildeu} from \eqref{ie:psv}-\eqref{ie:pssv}.
\end{proof}

\subsection{The solution to the dual variational inequality \eqref{eq:VIu}}

Using the transform $y=e^s, t=T-\tau, \hat u(y,t)=\tilde u(s,\tau)$, where $\tilde u(s,\tau)$ is the solution to problem \eqref{eq:VItildeu}, we first show that $\hat u(y,t)$ is the solution to problem \eqref{eq:VIhatu}.

\begin{proposition}\label{thm:u}
$\hat u(y,t)=\tilde u(s,\tau)$ is the unique solution to problem \eqref{eq:VIhatu},
In particular, let $ \mathbb Q=(0,y_0(t))\times(0,T]$ with $y_0(t)=e^{Z(T-t)}$, and
\begin{equation}\label{def:u}
u(y,t)=\hat u(y,t),\quad (y,t)\in\mathbb Q.
\end{equation}
Then $u(y,t)$ is the unique solution to problem \eqref{eq:VIu} and
$e^{S(T-t)}\in C[0,T]$ is the free boundary to problem \eqref{eq:VIu} such that
\begin{eqnarray}
 \mathcal G_1&=& \big\{(y,t)|(1-p)u+py\partial_y u=0\big\}\notag \\
&=& \big\{(y,t)|0<y\leq e^{S(T-t)},t\in[0,T)\big\}, \label{fb:u1} \\
 \mathcal E_1&=& \big\{(y,t)|(1-p)u+py\partial_y u<0\big\}\notag \\
&=& \big\{(y,t)|e^{S(T-t)}<y\leq e^{Z(T-t)},t\in[0,T)\big\}. \label{fb:u2}
\end{eqnarray}
Moreover, $u(y,t)\in C^{2,1}(\mathbb Q)\cap C(\overline{\mathbb Q})$ that satisfies
\begin{eqnarray}
&&\partial_{y} u(y,t)<0, \quad (y,t)\in\mathbb Q,\label{ie:pyu}\\
&&\partial_{yy} u(y,t)>0, \quad (y,t)\in\mathbb Q.\label{ie:pyyu}
\end{eqnarray}
\end{proposition}
\begin{proof}
First, it is obvious that $\hat u(y,t)$ is the solution to problem \eqref{eq:VIhatu}.
It follows from \eqref{fb:tildeu0}-\eqref{fb:tildeu2} that
\begin{eqnarray*}
&&\max\left\{\partial_t\hat u + \frac{1}{2}\frac{\mu^2}{\sigma^2}y^2\partial_{yy}\hat u + \delta y\partial_y\hat u - \delta\hat  u-f(y),(1-p)\hat u+py\partial_y\hat u\right\}=0,\quad (y,t)\in \mathbb Q, \\ [2mm] 
&&\hat u(e^{Z(T-t)},t)=\partial_y\hat u(e^{Z(T-t)},t)=0.
\end{eqnarray*}
Combining with the terminal condition on $t=T$, we deduce that $u(y,t)$ given in \eqref{def:u} is the unique solution to problem \eqref{eq:VIu} in $\mathbb Q$ with $y_0(t)=e^{Z(T-t)}$.

Note that
$$(p-1)\tilde u-p\partial_s\tilde u=(p-1)u-py\partial_y u,$$
we can derive \eqref{fb:u1}-\eqref{fb:u2} from \eqref{fb:tildeu1}-\eqref{fb:tildeu2}.

The regularity of $\tilde u(s,\tau)$ in Theorem \ref{thm:tildeu} yields the regularity of $u(y,t)$ directly. 
According to \eqref{ie:pstildeu}-\eqref{ie:psstildeu}, we have
\begin{eqnarray*}
&&\partial_{y}u=e^{-s}\partial_s\tilde u(s,\tau)<0,\quad  (y,t)\in\mathbb Q,\\
&&\partial_{yy}u=e^{-2s}(\partial_{ss}\tilde u(s,\tau)-\partial_s\tilde u(s,\tau))>0,\quad  (y,t)\in\mathbb Q,
\end{eqnarray*}
which imply \eqref{ie:pyu}-\eqref{ie:pyyu}.
\end{proof}

\section{Proof of Main Results}\label{sec:control}
\setcounter{equation}{0}

\subsection{The solution to the HJB variational inequality \eqref{eq:VIU}}
\begin{theorem}\label{thm:U}
There exists a  solution $U(\omega,t)$ to the problem \eqref{eq:VIU}. Moreover, $U(\omega,t)\in C^{2,1}(\mathcal Q)$ that satisfies
\begin{eqnarray}
&&\partial_{\omega} U(\omega,t)>0, \quad (\omega,t)\in\mathcal Q,\label{ie:monotonic}\\
&&\partial_{\omega\omega} U(\omega,t)<0, \quad (\omega,t)\in\mathcal Q. \label{ie:concave}\\
 &&0\leq U(\omega,t)\leq\frac{\omega^{1-p}}{1-p},\quad (\omega,t)\in\mathcal Q, \label{ie:Uesti}
\end{eqnarray}
Moreover, if $U_1,U_2$ are the solutions to problem \eqref{eq:VIU} satisfying \eqref{ie:Uesti} and
\begin{equation}\label{U1U2}
\lim\limits_{\omega\rightarrow0^+}\omega^{p-1}(U_1-U_2)=0,
\end{equation}
then $U_1=U_2$ in $\mathcal Q$.
\end{theorem}
\begin{proof}
Using the fact that $u(y,t)\in C^{2,1}(\mathbb Q)$ and \eqref{ie:pyyu}, we can deduce the existence of a continuous inverse function of $I(y,t)=-\partial_yu(y,t)=\omega>0$ that
$I^{-1}(\omega,t)\in C(\mathcal Q).$ It follows from \eqref{ie:pyu} and the boundary condition on $y_0(t)$ that we have $y_0(t)=I^{-1}(0,t)$. 
Set
$$
U(\omega,t):=u(I^{-1}(\omega,t),t)+\omega I^{-1}(\omega,t). 
$$
%
Then $U(\omega,t)$ is continuous. Moreover, we have
\begin{eqnarray*}
&&\partial_{\omega} U(\omega,t)=\partial_{y} u\partial_\omega(I^{-1}(\omega,t)) + I^{-1}(\omega,t) + \omega \partial_\omega(I^{-1}(\omega,t)) \\
&&\hspace{0.7in}=\partial_{y} u\frac{1}{\partial_y I}+I^{-1}+I\frac{1}{\partial_y I} \\
&&\hspace{0.7in}=\frac{\partial_{y} u}{-\partial_{yy} u}+y+\frac{-\partial_{y} u}{-\partial_{yy} u}=y=I^{-1}(\omega,t)\in C(\mathcal Q), \\
&&\partial_{\omega\omega} U(\omega,t)=-\frac{1}{\partial_{yy} u(y,t)}\in C(\mathcal Q), \\
&&\partial_t U(\omega,t) = \partial_t u(y,t) = \partial_t u(I^{-1}(\omega,t),t)\in C(\mathcal Q), \\
&&U(0,t)=u(y_0(t),t)=0, \\
&&U(\omega,T)=\frac{1}{1-p}\big(I^{-1}(\omega,T)\big)^{1-\frac{1}{p}}=\frac{1}{1-p}\omega^{1-p}. 
\end{eqnarray*}
Hence, $U(\omega,t)\in C^{2,1}(\mathcal Q)$ is a solution to problem  \eqref{eq:VIU} using the dual transform \eqref{eq:yomega}-\eqref{eq:Uu} and the fact that $u(y,t)$ is the unique solution of \eqref{eq:VIu}.
As $\partial_{\omega} U(\omega,t)=y>0$, we have \eqref{ie:monotonic}. In addition, \eqref{ie:pyyu} implies the desired result \eqref{ie:concave}. 

We then show the estimate \eqref{ie:Uesti}. By its definition, we know $V(x,z,t)\geq0$ and 
$$U(\omega,t)=\frac{1}{z^{1-p}}V(x,z,t)\geq0.$$ 
Next we will show the right hand side of \eqref{ie:Uesti}. In view of \eqref{ie:ptauv} and the fact
$$u(y,t)=y^{\frac{p-1}{p}}v(s,\tau)=y^{\frac{p-1}{p}}v(\ln y,T-t),$$
then $$\partial_t u(y,t)=-y^{\frac{p-1}{p}}\partial_\tau v(\ln y,T-t)\geq0.$$
Hence $$\partial_t U(\omega,t)  = \partial_t u(I^{-1}(\omega,t),t)\geq0,$$
together with the terminal condition $U(\omega,T)=\frac{\omega^{1-p}}{1-p}$, we obtain
$$U(\omega,t)\leq\frac{\omega^{1-p}}{1-p},\quad 0\leq t\leq T.$$

In what follows, we show the uniqueness of the solution to problem \eqref{eq:VIU} by the contradiction argument. 
Suppose $U_1, U_2$ are two distinct solutions to problem \eqref{eq:VIU} satisfying 
\eqref{U1U2}  that
 $\mathcal N:=\{U_1(\omega,t)>U_2(\omega,t)\}\neq\emptyset$, where
\begin{eqnarray*}
&&\mathcal N_1:=\{(\omega,t)\in\mathcal N | (1-p)U_1(\omega,t)-\omega\partial_\omega U_1< (1-p)U_2(\omega,t)-\omega\partial_\omega U_2\},\\
&&\mathcal N_2:=\{(\omega,t)\in\mathcal N | (1-p)U_1(\omega,t)-\omega\partial_\omega U_1\geq (1-p)U_2(\omega,t)-\omega\partial_\omega U_2\}.
\end{eqnarray*}
It follows from the definition of $\mathcal N_1$ that
\begin{eqnarray*}
&& \partial_{t}U_1-\frac 1 2\frac{\mu^2}{\sigma^2}\frac{( \partial_{\omega}U_1)^2}{ \partial_{\omega\omega} U_1}+\frac{(\hat{c}^*(\omega,t))^{1-p}}{1-p}-{\hat c}^*(\omega,t) \partial_\omega U_1-\delta U_1=0,\\
&& \partial_{t}U_2-\frac 1 2\frac{\mu^2}{\sigma^2}\frac{( \partial_{\omega}U_2)^2}{ \partial_{\omega\omega} U_2}+\frac{(\hat{c}^*(\omega,t))^{1-p}}{1-p}-{\hat c}^*(\omega,t) \partial_\omega U_2-\delta U_2\leq 0.
\end{eqnarray*}
Denote $c^*(\omega,t)=g(\partial_\omega U)$ and $\tilde u=U_1-U_2$. We have that
\begin{eqnarray*}
\frac{(\partial_\omega U_1)^2}{\partial_{\omega\omega} U_1}-\frac{(\partial_\omega U_2)^2}{\partial_{\omega\omega} U_2} 
&=&\frac{(\partial_\omega U_1)^2}{\partial_{\omega\omega} U_1}-\frac{(\partial_\omega U_1)^2}{\partial_{\omega\omega} U_2}+\frac{(\partial_\omega U_1)^2}{\partial_{\omega\omega} U_2}-\frac{(\partial_\omega U_2)^2}{\partial_{\omega\omega} U_2}\\
&=&-\frac{(\partial_\omega U_1)^2}{\partial_{\omega\omega} U_1\partial_{\omega\omega} U_2}\partial_{\omega\omega}\tilde u+\frac{\partial_{\omega} U_1+\partial_\omega U_2}{\partial_{\omega\omega} U_2}\partial_{\omega}\tilde u,
\end{eqnarray*}
then it holds that
\begin{eqnarray}\label{eq:11}
 &&\partial_{t}\tilde u +\frac{(\partial_\omega U_1)^2}{\partial_{\omega\omega} U_1\partial_{\omega\omega} U_2} \partial_{\omega\omega} \tilde u+\frac{\partial_{\omega} U_1+\partial_\omega U_2}{\partial_{\omega\omega} U_2}\partial_{\omega}\tilde u -\delta \tilde u\nonumber\\
& & \geq g(\partial_\omega U_1)\partial_\omega U_1-g(\partial_\omega U_2)\partial_\omega U_2-\frac{[g(\partial_\omega U_1)]^{1-p}}{1-p}+\frac{[g(\partial_\omega U_2)]^{1-p}}{1-p}.
\end{eqnarray}
We next show the right hand side of \eqref{eq:11} is non-negative. Denote 
$$G(x):=g(x)x-\frac{[g(x)]^{1-p}}{1-p}.$$
It holds in $\mathcal N_1$ that
$$
0<(1-p)\tilde u<\omega\partial_\omega\tilde u,\quad(\omega,t)\in\mathcal N_1. 
$$
Hence, we obtain $\partial_\omega U_1-\partial_\omega U_2=\partial_\omega \tilde u>0$. In view that
\begin{eqnarray*}
&&g(\partial_\omega U)=c^*(\omega,t)=\left\{\begin{array}{ll}
1,\quad \partial_\omega U<1,\\
(\partial_\omega U)^{-1/p},\quad 1<\partial_\omega U<\alpha^{-p},\\
\alpha,\quad \partial_\omega U>\alpha^{-p},
\end{array}
\right.\\
&&[g(x)]^{-p}=\left\{\begin{array}{ll}
1,\quad x<1,\\
x,\quad 1<x<\alpha^{-p},\\
\alpha^{-p},\quad x>\alpha^{-p},
\end{array}
\right.\\
&&g'(x)=\left\{
\begin{array}{ll}
0,\\
-\frac1p x^{-\frac1p-1},\\
0,
\end{array}
\right.
\end{eqnarray*}
and the facts that $g'(x)\leq0$ and
$$
G'(x)=g'(x)x+g(x)-[g(x)]^{-p}g'(x)=g'(x)[x-[g(x)]^{-p}]+g(x)=g(x)\geq0,
$$
we deduce that $G(x)$ is increasing in $x$. It then follows that $G(\partial_\omega U_1)\geq G(\partial_\omega U_2)$, and the right hand side of \eqref{eq:11} is nonnegative. Therefore, $\tilde u$ satisfies the following linear equation
$$
\partial_t \tilde u-\mathcal L^*\tilde u\geq0,\quad (\omega,t)\in\mathcal N_1, 
$$ 
where the coefficients in the operator $\mathcal L^*$ are all determined. To apply the maximum principle, we need the boundary conditions on the parabolic boundary $\partial_p\mathcal N_1$. By the definition of $\mathcal N_1$, we have
$$(1-p)\tilde u-\omega\partial_\omega\tilde u=0,\quad (\omega,t)\in\partial_p\mathcal N_1.$$

\begin{center}
\begin{picture}(200,100)(0,0)
\put(10,10){\vector(1,0){180}} \put(20,0){\vector(0,1){105}}
\qbezier[80](20, 42)(94, 42)(104, 42)
\put(12,95){$t$}
\put(180,2){\footnotesize $\omega$}
\put(90,47){\footnotesize ${B}$}
\put(10,47){\footnotesize ${A}$}
\put(105,47){\footnotesize ${C}$}
\put(180,40){\footnotesize ${\partial_p\mathcal N}$}
\put(140,75){\footnotesize ${\partial_p\mathcal N_1}$}
\put(100,60){\footnotesize ${\mathcal N_1}$}
\put(145,50){\footnotesize ${\mathcal N_2}$}
\put(18, 40){$\bullet$}
\put(105, 40){$\bullet$}
\put(90, 40){$\bullet$}
\qbezier(20,30)(80,22)(182,40)
\qbezier(50,60)(70,42)(102,42)
\qbezier(100,42)(124,40)(150,72)
\end{picture} \bigskip
\ \\
\text{Figure 4}
 \label{fig1}\end{center}
 
Define $t_0=\inf\{t:(\omega,t)\in\mathcal N_1\}$. Denote $B(\omega_0,t_0), C(\omega_1,t_0)\in\partial_p\mathcal N_1$, where $\omega_0=\inf\{\omega:(\omega,t_0)\in\partial_p\mathcal N_1\},\;\omega_1\geq\omega_0$ and $\overline{BC}\subset\partial_p\mathcal N_1$, as illustrated in the above figure. We claim that $\tilde u(B)=0$. If it is not true, then we consider $A(0,t_0)$. So we have
$$
(1-p)\tilde u-\omega \partial_\omega\tilde  u\geq0,\quad(\omega,t)\in\overline{AB}.
$$
It follows that
$$\partial_\omega(\omega^{p-1}\tilde u)\leq0,\quad(\omega,t)\in\overline{AB}.$$
By the condition \eqref{U1U2}, we have  $\omega^{p-1}\tilde u(B)\leq 0$. Together with $\omega^{p-1}\tilde u(B)\geq 0$, we obtain $\tilde u(B)=0$, and hence
$$\tilde u(\omega,t)=0,\quad(\omega,t)\in\overline{BC}. $$
By the maximum principle, we know
$$U_1-U_2=\tilde u\leq0, \quad (\omega,t)\in \mathcal N_1,$$
which leads to a contradiction with the definition of $\mathcal N_1$. Hence, $\mathcal N_1=\emptyset$, and $\mathcal N=\mathcal N_2$.

In view that
\begin{eqnarray*}
\left\{\begin{array}{ll}
\partial_\omega(\omega^{p-1}\tilde u)\geq0,\quad(\omega,t)\in\partial_p\mathcal N,\\
\omega^{p-1}\tilde u=0,\quad(\omega,t)\in\partial_p\mathcal N,
\end{array}\right.
\end{eqnarray*}
we have  $\omega^{p-2}\tilde u=0$, for $(\omega,t)\in\mathcal N$, which contradicts the definition of $\mathcal N$. The uniqueness of the solution to problem \eqref{eq:VIU} then follows.
\end{proof}


\subsection{Proof of Theorem \ref{thm:V}}

First, the next result gives the analytical characterization of free boundaries $\omega^*(t),\omega_1(t)$ and $\omega_\alpha(t)$ in Theorem \ref{thm:V}.

\begin{theorem}\label{thm:Ufb}
Let $w(s,\tau)$ be the solution to problem \eqref{eq:VIw},  and $S(\tau)$ be defined in \eqref{freeS}. There exist three free boundaries $\omega^*(t),\omega_1(t)$ and $\omega_\alpha(t)$ to problem \eqref{eq:VIU} such that
\begin{eqnarray}
&&\omega_\alpha(t)<\omega_1(t)\leq\omega^*(t),\quad t\in[0,T],\label{eq:alpha1*} \\
 &&JR = \big\{(\omega,t)\in\overline{\mathcal Q} ~ | ~ (1-p)U-\omega\partial_\omega U=0\big\} \notag \\
&&\hspace{0.2in} = \big\{(\omega,t)\in\overline{\mathcal Q} ~ | ~ \omega\geq \omega^*(t),t\in[0,T)\big\}, \label{fb:U1} \\
 && CR = \big\{(\omega,t)\in\overline{\mathcal Q} ~ | ~ (1-p)U-\omega\partial_\omega U<0\big\} \notag \\
&&\hspace{0.2in} = \big\{(\omega,t)\in\overline{\mathcal Q} ~ | ~ \omega<\omega^*(t),t\in[0,T)\big\}.\label{fb:U2} 
\end{eqnarray}
 Moreover, for $t\in[0,T)$, we have the analytical form in terms of $w(s,\tau)$  that
\begin{eqnarray}
&&\omega^*(t)=\frac{1-p}{p}e^{-\frac1pS(T-t)}\int_{Z(T-t)}^{S(T-t)}w(\xi,T-t)d\xi, \label{solvo*} \\ 
&&\omega_1(t)=\frac{1-p}{p}\int_{Z(T-t)}^{0}w(\xi,T-t)d\xi 
-w(0,T-t), \label{solvo1}\\
&&\omega_\alpha(t)=\alpha\frac{1-p}{p}\int_{Z(T-t)}^{-p\ln\alpha}w(\xi,T-t)d\xi 
-\alpha w(-p\ln\alpha,T-t). \label{solvoa}
\end{eqnarray}
Here,  $\omega^*(t), \omega_1(t), \omega_\alpha(t)$ are continuous in $t$. In particular, our conjectured free boundary $\omega^*(t)$ in Remark \ref{rmk:omega} is now characterized analytically by \eqref{solvo*}.

In addition, the candidate optimal feedback control ${\hat c}^*(\omega,t)$ given in \eqref{eq:defc} satisfies that
\begin{eqnarray}
&&{\hat c}^*(\omega,t)=\left\{\begin{array}{ll}
\alpha, & \mbox{if }  0<\omega\leq \omega_\alpha(t), \\
(\partial_\omega U)^{-\frac{1}{p}}, & \mbox{if } \omega_\alpha(t)<\omega< \omega_1(t), \\
1, & \mbox{if } \omega_1(t)\leq\omega\leq \omega^*(t). \\ 
\end{array}\right.\label{eq:solc}
\end{eqnarray}

\end{theorem}
\begin{proof}
First, by \eqref{eq:defu}-\eqref{eq:uy}, we have
\begin{eqnarray*}
(1-p)u(y,t)+py\partial_y u(y,t)
=(1-p)U(\omega,t)-\omega\partial_\omega U(\omega,t).
\end{eqnarray*}
According to \eqref{ie:pyyu}, $I(y,t)=-\partial_y u(y,t)$ is strictly decreasing in $y$, it follows that
\begin{eqnarray*}
& & \mathcal{G}_1 = \big\{(y,t)\in\mathbb Q ~ | ~ I(y,t)\geq I(e^{S(T-t)},t), t\in[0,T)\big\}, \\
& & \mathcal{E}_1 = \big\{(y,t)\in\mathbb Q ~ | ~ I(y,t)< I(e^{S(T-t)},t), t\in[0,T)\big\}.
\end{eqnarray*}
The above results, together with \eqref{fb:u1}-\eqref{fb:u2}, imply the existence of a free boundary $\omega^*(t)=I(e^{S(T-t)},t)=-\partial_y u(e^{S(T-t)},t)$ such that \eqref{fb:U1}-\eqref{fb:U2} hold true. 

Using the transform $u(y,t)=\tilde u(s,\tau)=e^{-\frac{1-p}{p}s}v(s,\tau)$, $y=e^s$ and $\tau=T-t$, we have that
\begin{eqnarray*}
\omega^*(t)&=&-\partial_yu(e^{S(T-t)},t)=-e^{-S(T-t)}\partial_s\tilde u(S(T-t),T-t)\\
&=&e^{-\frac1pS(T-t)}\left[\frac{1-p}{p}v(S(T-t),T-t)-\partial_sv(S(T-t),T-t)\right]\\
&=&\frac{1-p}{p}e^{-\frac1pS(T-t)}\int_{Z(T-t)}^{S(T-t)}w(\xi,T-t)d\xi,
\end{eqnarray*}
where the last equality is due to the expression of $v(s,\tau)$ in \eqref{eq:solv} and the fact that $\partial_sv(S(T-t),T-t)=0$.

Next, we show \eqref{eq:solc}. In the domain $CR$, we have
$$\partial_tU-\frac 1 2\frac{\mu^2}{\sigma^2}\frac{(\partial_{\omega}U)^2}{\partial_{\omega\omega}U}+\frac{\hat c^{1-p}(\omega,t)}{1-p}-\hat c(\omega,t)\partial_\omega U-\delta U=0,$$
where $\hat c(\omega,t)$ is given by \eqref{eq:defc}. Then, let us consider two free boundaries $\omega_1(t)$ and $\omega_\alpha(t)$ in problem \eqref{eq:VIU} satisfying that 
\begin{eqnarray}\label{o1oa}
\left\{\begin{array}{ll}
\partial_\omega U(\omega_1(t),t)=1,\\
\partial_\omega U(\omega_\alpha(t),t)=\alpha.
\end{array}\right.
\end{eqnarray}
By the strict concavity \eqref{ie:concave} of $U(\omega,t)$, we get the desired result \eqref{eq:alpha1*}. 
Moreover, the strict concavity of $U(\omega,t)$, together with definition \eqref{eq:defc} of $\hat c(\omega,t)$, implies the expression \eqref{eq:solc}.

Next, we continue to show \eqref{solvo1} and \eqref{solvoa}.
In view of \eqref{eq:omegay} and \eqref{eq:uy}, for any fixed $t\in[0,T)$,  the first equation of \eqref{o1oa} implies that $\omega_1(t)$ satisfies
\begin{eqnarray*}
\omega_1(t)&=&(\partial_\omega U)^{-1}(1,t)=-\partial_yu(1,t)=-\partial_s\tilde u(0,T-t) \\
&=&\frac{1-p}{p}v(0,T-t)-\partial_sv(0,T-t) \\
&=&\frac{1-p}{p}\int_{Z(T-t)}^0w(\xi,\tau)d\xi-w(0,T-t).
\end{eqnarray*}
Similarly, it follows from the second equation of \eqref{o1oa} that $\omega_\alpha(t)$ satisfies
\begin{eqnarray*}
\omega_\alpha(t)&=&(\partial_\omega U)^{-1}(\alpha^{-p},t)=-\partial_yu(\alpha^{-p},t)=-\alpha^{p}\partial_s\tilde u(-p\ln\alpha,T-t) \\
&=&\alpha\left[\frac{1-p}{p}v(-p\ln\alpha,T-t)-\partial_sv(-p\ln\alpha,T-t)\right] \\
&=&\alpha\frac{1-p}{p}\int_{Z(T-t)}^{-p\ln\alpha}w(\xi,T-t)d\xi-\alpha w(-p\ln\alpha,T-t).
\end{eqnarray*}

The continuity of $\omega^*(t), \omega_1(t), \omega_\alpha(t)$ can be deduced from the expressions in \eqref{solvo*}-\eqref{solvoa}, the continuity of $w(s,\tau)$ and $S(\tau)$ with respect to $\tau$.

%
%
\end{proof}

\begin{corollary}\label{thm:o*alpha}
The free boundaries $\omega^*(t)$ and $\omega_\alpha(t)$ given in \eqref{solvo*} and \eqref{solvoa}, respectively, increase in $\alpha$.
\end{corollary}
\begin{proof}
%
%
%
%

Since $\omega^*(t)=-\partial_yu(e^{S(T-t)},t)$,  by \eqref{ie:pyyu} and Lemma \ref{thm:fbde}
it then follows  that
\begin{eqnarray*}
\frac{\partial\omega^*(t)}{\partial\alpha}&=&-\partial_{yy}u(e^{S(T-t)},t)e^{S(T-t)}\frac{\partial S(T-t)}{\partial\alpha}>0. 
\end{eqnarray*}
Hence, $\omega^*(t)$ increases in $\alpha$.

As $\omega_\alpha(t)=-\partial_yu(\alpha^{-p},t)$, we deduce that
\begin{eqnarray*}
\frac{\partial\omega_\alpha(t)}{\partial\alpha}&=&-\partial_{yy}u(\alpha^{-p},t)(-p\alpha^{-p-1})=p\alpha^{-p-1}\partial_{yy}u>0. 
\end{eqnarray*}
Hence, $\omega_\alpha(t)$ increases in $\alpha$.
\end{proof}

Building upon that $U(\omega,t)$ is strictly concave and strictly increases in $\omega$ in Theorem \ref{thm:U}, we can see that problem \eqref{eq:VIU} is equivalent to problem \eqref{eq:HJBU}. 
Let $V(x,z,t)=z^{1-p}U(\frac{x}{z},t)$, we can then show that $V(x,z,t)$ is the solution to problem \eqref{HJBV}. 


\begin{proof}[Proof of Theorem \ref{thm:V}]
By Theorem \ref{thm:U} and Theorem  \ref{thm:Ufb}, it is straightforward to check that $V(x,z,t)$ is the unique solution in $C^{2,1}(Q)\cap C(\bar Q)$ to problem  \eqref{HJBV} and also \eqref{eq:pxxV} holds. In view of \eqref{fb:U1}-\eqref{fb:U2}, if $x\geq \omega^*(t)z$, then
$$\partial_zV(x,z,t)=z^{-p}[(1-p)U(\omega,t)-\omega\partial_{\omega}U(\omega,t)]=0,$$
for $\omega = \frac{x}{z}\geq \omega^*(t)$; if $x<\omega^*(t)z$, then
$$\partial_zV(x,z,t)=z^{-p}[(1-p)U(\omega,t)-\omega\partial_{\omega}U(\omega,t)]<0,$$
for $\omega = \frac{x}{z} <\omega^*(t)$. 
Hence, we obtain results \eqref{fb:V1}-\eqref{fb:V2}. 

To prove the optimality of feedback control \eqref{solpi*} and \eqref{solc*}, it suffices to show that the following two conditions hold for all $x,z\in \mathbb{R}^+$ and $t\leq T$:
\begin{itemize}
    \item[(i)] The following SDE admits a unique strong solution $(X^*_s)_{s\geq t}$ that
    \begin{equation}\label{sde*}
    \left\{
        \begin{array}{ll}
            dX^*_s= &  \left(\mu \pi^*\left(X^*_s,\frac{M^*_s}{\omega^*(s)},s\right)-c^*\left(X^*_s,\frac{M^*_s}{\omega^*(s)},s\right)\right)ds+\sigma \pi^*\left(X^*_s,\frac{M^*_s}{\omega^*(s)},s\right)dW_s;\\
            M^*_s= & \max\left\{\omega^*(t)z,\sup_{t\leq u\leq s}X^*_u\right\};\\
            X^*_t=&x.
        \end{array}
        \right.
    \end{equation}
    Furthermore, the feedback controls $(\pi^*_t,c^*_t):=(\pi^*(X^*_t,z^*_t,t), c^*(X^*_t,z^*_t,t))$ are admissible.
    \item[(ii)]  It holds that
    \begin{eqnarray}
    V(x,z,t)
    \hspace{-0.3cm}&=&\hspace{-0.3cm}
    \mathbb{E}\left[\int_t^{T\wedge\tau}e^{-\delta(s-t)}\frac{(c^*_s)^{1-p}}{1-p}ds+e^{-\delta(T\wedge \tau-t)}\frac{(X^*_{T\wedge\tau})^{1-p}}{1-p}\bigg|X^*_t=x,z^*_t=z\right]
    \nonumber\\
    \hspace{-0.3cm}&=&\hspace{-0.3cm}
    \sup_{(\pi,c)\in\mathcal{A}(x)}\mathbb{E}\left[\int_t^{T\wedge\tau}e^{-\delta(s-t)}\frac{c_s^{1-p}}{1-p}ds+e^{-\delta(T\wedge \tau-t)}\frac{X_{T\wedge\tau}^{1-p}}{1-p}\bigg|X_t=x,z_t=z\right].\nonumber
    \end{eqnarray}
\end{itemize}
For simplicity, we write $\mathbb{E}[\cdot|X_t=x,z_t=z]$ as $\mathbb{E}[\cdot]$ in the following. 

We first prove that condition (i) holds. By the form of $\pi^*$, $c^*$ and Remark \ref{rmk:omega}, we know that the consumption running maximum process $z^*_s$ satisfies $z^*_s=\frac{M^*_s}{\omega^*(s)}$ for all $s\in[t,T\wedge\tau]$ with $M^*_s$ given by \eqref{sde*}. Hence, if the processes $X^*_s$ and $z^*_s$ satisfying \eqref{sde*} exist, then it is easy to check that $(\pi^*_t,c^*_t)=(\pi^*(X^*_t,z^*_t,t), c^*(X^*_t,z^*_t,t))$ are admissible. To show that $\eqref{sde*}$ has a unique strong solution, by Theorem 7 in section 3 of Chapter 5 of \cite{PE}, one needs to show that the functionals
\begin{align*}
G(s,X):=\pi^*\left(X_s,\frac{1}{\omega^*(t)}\max\left\{\omega^*(t)z,\sup_{t\leq u\leq s}X_u\right\},s\right)
\end{align*}
and
\begin{align*}
F(s,X):=c^*\left(X_s,\frac{1}{\omega^*(t)}\max\left\{\omega^*(t)z,\sup_{t\leq u\leq s}X_u\right\},s\right)
\end{align*}
defined for $s\geq t$ and for continuous functions $X:\mathbb{R}^+\rightarrow\mathbb{R}^+$, are functional Lipschitz in the sense of \cite{PE}.
Note that one can apply the similar arguments as used in Appendix A of \cite{ABY19} to prove the Lipschitz property of the feedback functions $\pi^*$ and $c^*$.
Therefore, for any $s\in[t,T\wedge\tau]$ and continuous functions $X$ and $Y$, we have
\begin{eqnarray}
|G(s,X)-G(s,Y)|
\hspace{-0.3cm}&\leq&\hspace{-0.3cm} K\left[|X_s-Y_s|+\left|\max\left\{\omega^*(t)z,\sup_{t\leq u\leq s}X_u\right\}-\max\left\{\omega^*(t)z,\sup_{t\leq u\leq s}Y_u\right\}\right|\right]
\nonumber\\
\hspace{-0.3cm}&\leq&\hspace{-0.3cm} 
2K\sup_{t\leq u\leq s}|X_u-Y_u|.\nonumber
\end{eqnarray}
Hence, $G$ is functional Lipschitz. Similarly, $F$ is also functional Lipschitz.

To prove condition (ii), let us introduce a sequence of stopping times as
$$\tau_n:=n\wedge T\wedge\tau \wedge\inf\left\{s>t; \int_t^se^{-\delta(u-t)}(\pi^*_u)^2(\partial_xV(X^*_u,z^*_u,u))^2du\geq n\text{ or }z^*_s\geq n \right\},\quad n\geq t,$$ with $\tau_n\rightarrow T\wedge\tau$ as $n\rightarrow\infty$. By an application of It\^o's formula to $e^{-\delta(s-t)}V(X^*_{s},z^*_{s},s)$ on $[t,\tau_n]$, we have
\begin{eqnarray}
\hspace{-0.3cm}&&\hspace{-0.3cm}
e^{-\delta(\tau_n-t)}V(X^*_{\tau_n},z^*_{\tau_n},\tau_n)
\nonumber\\
\hspace{-0.3cm}&=&\hspace{-0.3cm}
V(x,z,t)-\int_t^{\tau_n}e^{-\delta(s-t)}\frac{(c^*_s)^{1-p}}{1-p}ds+\int_t^{\tau_n}e^{-\delta(s-t)}\partial_tV(X^*_s,z^*_s,s)ds
\nonumber\\
\hspace{-0.3cm}&&\hspace{-0.3cm}
+\int_t^{\tau_n}e^{-\delta(s-t)}\left[\frac{1}{2}\sigma^2(\pi^*_s)^2\partial_{xx}V(X^*_s,z^*_s,s)+(\mu\pi^*_s-c^*_s)\partial_xV(X^*_s,z^*_s,s)\right.
\nonumber\\
\hspace{-0.3cm}&&\hspace{-0.3cm}
\left.-\delta V(X^*_s,z^*_s,s)+\frac{(c^*_s)^{1-p}}{1-p}\right]ds+\int_t^{\tau_n}e^{-\delta(s-t)}\sigma\pi^*_s\partial_xV(X^*_s,z^*_s,s)dW_s
\nonumber\\
\hspace{-0.3cm}&&\hspace{-0.3cm}
+\int_t^{\tau_n}e^{-\delta(s-t)}\partial_zV(X^*_s,z^*_s,s)d(z^*_s)^c
+\sum_{t\leq s\leq \tau_n}e^{-\delta(s-t)}\left(V(X^*_s,z^*_{s},s)-V(X^*_s,z^*_{s-},s)\right),\nonumber
\end{eqnarray}
where $(z^*_s)^c$ stands for the continuous part of $z^*_s$. Taking expectations on both sides, we obtain
\begin{align}
\label{V.dpp}
&V(x,z,t)
\nonumber\\
&
=\mathbb{E}\left[\int_t^{\tau_n}e^{-\delta(s-t)}\frac{(c^*_s)^{1-p}}{1-p}ds+e^{-\delta(\tau_n-t)}V(X^*_{\tau_n},z^*_{\tau_n},\tau_n)\right]
\nonumber\\
&
-\mathbb{E}\left[\int_t^{\tau_n}e^{-\delta(s-t)}\left[\frac{1}{2}\sigma^2(\pi^*_s)^2\partial_{xx}V(X^*_s,z^*_s,s)
\right.\right.\nonumber\\
&
\left.\left.
+(\mu\pi^*_s-c^*_s)\partial_xV(X^*_s,z^*_s,s)
-\delta V(X^*_s,z^*_s,s)+\frac{(c^*_s)^{1-p}}{1-p}\right]ds\right]
\nonumber\\
&
-\mathbb{E}\left[\int_t^{\tau_n}e^{-\delta(s-t)}\partial_zV(X^*_s,z^*_s,s)d(z^*_s)^c
+\sum_{t\leq s\leq \tau_n}e^{-\delta(s-t)}\left(V(X^*_s,z^*_{s},s)-V(X^*_s,z^*_{s-},s)\right)\right].
\end{align}
By plugging back the feedback controls in \eqref{solpi*} and \eqref{solc*} and the fact that $V$ is the solution to problem \eqref{HJBV}, for $(X^*_s,z^*_s,s)\in\mathcal{C}$, we have
\begin{align}\label{V.term1}
&\mathbb{E}\Bigg[\int_t^{\tau_n}e^{-\delta(s-t)}\Big[\frac{1}{2}\sigma^2(\pi^*_s)^2\partial_{xx}V(X^*_s,z^*_s,s)+(\mu\pi^*_s-c^*_s)\partial_xV(X^*_s,z^*_s,s)
\nonumber\\
&-\delta V(X^*_s,z^*_s,s)+\frac{(c^*_s)^{1-p}}{1-p}\Big]ds\Bigg]=0,
\end{align}
and hence
\begin{align}\label{V.term2}
\mathbb{E}\left[\int_t^{\tau_n}e^{-\delta(s-t)}\partial_zV(X^*_s,z^*_s,s)d(z^*_s)^c
+\sum_{t\leq s\leq \tau_n}e^{-\delta(s-t)}\left(V(X^*_s,z^*_{s},s)-V(X^*_s,z^*_{s-},s)\right)\right]=0.
\end{align}
For $(X^*_s,z^*_s,s)\in\mathcal{D}$ and satisfying $X^*_s=\omega^*(s)z_s$, by the continuity of $V$, $\partial_xV$ and $\partial_{xx}V$ as well as the fact that $\partial_zV(X^*_s,z^*_s,s)=0$, we have \eqref{V.term1} and \eqref{V.term2} hold. Furthermore, for $(X^*_s,z^*_s,s)\in\mathcal{D}$ and satisfying $X^*_s>\omega^*(s)z_{s-}$, by employing the feedback controls, it occur only at the initial time $s=t$ and $z^*_t$ immediately jumps from $z^*_{t-}$ to a new global maximum level $z^*_t=\frac{X^*_t}{\omega^*(t)}$, which together with the fact that $\partial_zV(X^*_s,z^*_s,s)=0$ yields \eqref{V.term1} and \eqref{V.term2}.
Then, it follows from \eqref{V.dpp} that
$$V(x,z,t)=\mathbb{E}\left[\int_t^{\tau_n}e^{-\delta(s-t)}\frac{(c^*_s)^{1-p}}{1-p}ds+e^{-\delta(\tau_n-t)}V(X^*_{\tau_n},z^*_{\tau_n},\tau_n)\right].$$
Letting $n\rightarrow\infty$ in above equation and using Monotone Convergence Theorem, we obtain
\begin{eqnarray}\label{conv.1}
\lim_{n\rightarrow\infty}\mathbb{E}\left[\int_t^{\tau_n}e^{-\delta(s-t)}\frac{(c^*_s)^{1-p}}{1-p}ds\right]=\mathbb{E}\left[\int_t^{T\wedge\tau}e^{-\delta(s-t)}\frac{(c^*_s)^{1-p}}{1-p}ds\right].
\end{eqnarray}

Additionally, note that the value function of the utility maximization problem on terminal wealth under a drawdown constraint with power utility function is less than the value function of Merton investment problem with power utility function ($\alpha=0$ in our context) that
$$\sup_{(\pi,c)\in\mathcal{A}(x)}\mathbb{E}\left[\frac{1}{1-p}X_{\tau_n}^{1-p}\right]\leq \sup_{(\pi,c)\in\mathcal{A}_0(x)}\mathbb{E}\left[\frac{1}{1-p}X_{\tau_n}^{1-p}\right],$$
where $\mathcal{A}_0(x)$ is the set of admissible strategies for Merton problem. Similar to Lemma A.3 of \cite{ABY19}, there exists a constant $M>0$ such that $0\leq V(x,z,t)\leq \frac{x^{1-p}}{1-p} M$ for all $x\geq 0$, independent of $z$ and $t$. Therefore, it holds that
$$\mathbb{E}\left[e^{-\delta(\tau_n-t)}V(X^*_{\tau_n},z^*_{\tau_n},\tau_n)\right]\leq M\mathbb{E}\left[e^{-\delta(\tau_n-t)}\frac{(X^*_{\tau_n})^{1-p}}{1-p}\right]\leq M
\sup_{(\pi,c)\in\mathcal{A}_0(x)}\mathbb{E}\left[\frac{X_{\tau_n}^{1-p}}{1-p}\right].
$$
Using the standard transversality condition in the Merton problem, the fact $V(x,z,t)\geq 0$, the terminal condition $V(x,z,T)=\frac{x^{1-p}}{1-p}$, the boundary condition $V(0,z,t)=0$, and the Dominated Convergence Theorem, we have that
\begin{eqnarray}\label{conv.2}
\lim_{n\rightarrow\infty}\mathbb{E}\left[e^{-\delta(\tau_n-t)}V(X^*_{\tau_n},z^*_{\tau_n},\tau_n)\right]
\hspace{-0.3cm}&=&\hspace{-0.3cm}
\mathbb{E}\left[e^{-\delta(T\wedge\tau-t)}V(X^*_{T\wedge\tau},z^*_{T\wedge\tau},T\wedge\tau)\right]
\nonumber\\
\hspace{-0.3cm}&=&\hspace{-0.3cm}
\mathbb{E}\left[e^{-\delta(T\wedge\tau-t)}\frac{(X^*_{T\wedge\tau})^{1-p}}{1-p}\right].  
\end{eqnarray}
Combining \eqref{conv.1} and \eqref{conv.2}, we deduce that the first equality in condition (ii) holds.

Finally, in view that the feedback controls $(\pi^*_t,c^*_t):=(\pi^*(X^*_t,z^*_t,t), c^*(X^*_t,z^*_t,t))$ are admissible, we readily obtain 
$$V(x,z,t)\leq \sup_{(\pi,c)\in\mathcal{A}(x)}\mathbb{E}\left[\int_t^{T\wedge\tau}e^{-\delta(s-t)}\frac{c_s^{1-p}}{1-p}ds+e^{-\delta(T\wedge \tau-t)}\frac{X_{T\wedge\tau}^{1-p}}{1-p}\right].$$

For the inverse inequality, repeating the arguments in the proof of \eqref{V.dpp} with any admissible strategies $(\pi,c)\in\mathcal{A}(x)$ and corresponding state process $(X,z)$, we have
\begin{align*}
&V(x,z,t)
\nonumber\\
&=\mathbb{E}\left[\int_t^{\tau_n}e^{-\delta(s-t)}\frac{(c_s)^{1-p}}{1-p}ds+e^{-\delta(\tau_n-t)}V(X_{\tau_n},z_{\tau_n},\tau_n)\right]
\nonumber\\
&-\mathbb{E}\Bigg[\int_t^{\tau_n}e^{-\delta(s-t)}\Big[\frac{1}{2}\sigma^2(\pi_s)^2\partial_{xx}V(X_s,z_s,s)+(\mu\pi_s-c_s)\partial_xV(X_s,z_s,s) \\
&-\delta V(X_s,z_s,s)+\frac{(c_s)^{1-p}}{1-p}\Big]ds\Bigg]
\nonumber\\
&-\mathbb{E}\left[\int_t^{\tau_n}e^{-\delta(s-t)}\partial_zV(X_s,z_s,s)d(z_s)^c
+\sum_{t\leq s\leq \tau_n}e^{-\delta(s-t)}\left(V(X_s,z_{s},s)-V(X_s,z_{s-},s)\right)\right].\nonumber
\end{align*}
Because $V(x,z,t)$ is the unique classical solution to problem \eqref{HJBV}, the second term and the third term on the right side of the above equation are non-negative, we obtain that
\begin{eqnarray}
V(x,z,t)\geq 
\mathbb{E}\left[\int_t^{\tau_n}e^{-\delta(s-t)}\frac{(c_s)^{1-p}}{1-p}ds+e^{-\delta(\tau_n-t)}V(X_{\tau_n},z_{\tau_n},\tau_n)\right].\nonumber
\end{eqnarray}
Letting $n\rightarrow\infty$ and applying the similar arguments in the proof of condition (i),  we conclude that the reverse inequality holds, which completes the proof.
\end{proof}

\section{Conclusions}\label{sec:con}
We revisit the optimal consumption problem under drawdown constraint formulated in \cite{ABY19} by featuring the finite investment horizon. For this stochastic control problem under control-type constraint, we contribute to the theoretical study on the existence and uniqueness of the classical solution to the parabolic HJB variational inequality. In particular, the consumption drawdown constraint induces some time-dependent free boundaries that deserve careful investigations. Using the dual transform and considering the auxiliary variational inequality with both function and gradient constraints, we develop some technical arguments to obtain the regularity of the unique solution as well as some analytical characterization of the associated time-dependent free boundaries such that the smooth fit conditions hold. As a result, we are able to derive and verify the optimal portfolio and consumption strategies in the piecewise feedback form.     

For future research extensions, it will be interesting to study other finite-time horizon optimal consumption problems when the utility function depends on the endogenous reference with respect to the past consumption maximum such as the formulation in \cite{DLPY} and \cite{LiYuZ2021arXiv}. Some new techniques are needed due to more wealth regimes and free boundary curves. It is also an appealing problem to study the finite-time horizon optimal consumption problems under habit formation constraint as investigated in \cite{Bah2022} and \cite{Bah2023}, where the dual transform can no longer linearize the variational inequality. New technical tools are needed to cope with the nonlinear dual parabolic variational inequality.

\ \\
\noindent
\textbf{Acknowledgements}: The authors sincerely thank anonymous referees for their valuable comments and suggestions that improved the paper significantly. X. Chen is supported by  NNSF of China  no.12271188.  X. Li is supported by Hong Kong RGC grants under no. 15216720 and 15221621. F. Yi is supported by NNSF of China  no.12271188 and 12171169.
X. Yu is supported by the Hong Kong RGC General Research Fund (GRF) under grant no. 15306523 and the Hong Kong Polytechnic University research grant under no. P0039251.

\end{document}